\documentclass[11pt,reqno]{amsart}

\usepackage[margin=1.2in]{geometry}
\usepackage{amsmath,amssymb,amsthm,amscd,mathrsfs,mathtools}
\usepackage{bbm}
\usepackage{esint}
\usepackage{cases}
\usepackage{cite}
\usepackage{enumerate}
\usepackage{graphicx}
\usepackage[dvipsnames]{xcolor}
\usepackage[implicit=true]{hyperref}
\usepackage[nameinlink,noabbrev]{cleveref}
\usepackage{microtype}

\numberwithin{equation}{section}

\newtheorem{theorem}{Theorem}[section]
\newtheorem{lemma}[theorem]{Lemma}
\newtheorem{remark}[theorem]{Remark}
\newtheorem{definition}[theorem]{Definition}
\newtheorem{proposition}[theorem]{Proposition}
\newtheorem{example}[theorem]{Example}
\newtheorem{corollary}[theorem]{Corollary}

\def\e{{\mathrm{e}}}
\def\eps{\varepsilon}

\def\p{\partial}

\def\<{{\langle}}
\def\>{{\rangle}}

\newcommand{\bx}{\begin{example}}\newcommand{\ex}{\end{example}}
\def\tr{{\rm tr}}

\def\dif{{\mathord{{\rm d}}}}

\def\no{\nonumber}
\def\={&\!\!=\!\!&}
\def\bt{\begin{theorem}}
\def\et{\end{theorem}}
\def\bl{\begin{lemma}}
\def\el{\end{lemma}}
\def\br{\begin{remark}}
\def\er{\end{remark}}
\def\bd{\begin{definition}}
\def\ed{\end{definition}}
\def\bp{\begin{proposition}}
\def\ep{\end{proposition}}
\def\bc{\begin{corollary}}
\def\ec{\end{corollary}}

\def\cB{{\mathcal B}}

\def\cG{{\mathcal G}}
\def\cH{{\mathcal H}}

\def\cO{{\mathcal O}}

\def\cQ{{Q}}

\def\cT{{\mathcal T}}

\def\mI{{\mathbb I}}

\def\mL{{\mathbb L}}

\def\mR{{\mathbb R}}

\def\mW{{\mathbb W}}

\def\bC{{\mathbf C}}

\def\sA{{\mathscr A}}

\def\sL{{\mathscr L}}

\def\sW{{\mathscr W}}

\def\geq{\geqslant}
\def\leq{\leqslant}

\def\div{\mathord{{\rm div}}}

\allowdisplaybreaks

\begin{document}

\title{Derivative estimates for SDEs with singular and unbounded coefficients}

\date{}

\author{Pengcheng Xia, Longjie Xie and Xicheng Zhang}

\address{Pengcheng Xia:
School of Mathematics and Big Data, Anhui University of Science and Technology, Huainan, Anhui 232001, China\\
Email: pcxia@whu.edu.cn
}

\address{Longjie Xie:
	School of Mathematics and Statistics, Jiangsu Normal University,
	Xuzhou, Jiangsu 221000, China\\
	Email: longjiexie@jsnu.edu.cn
}

\address{Xicheng Zhang: School of Mathematics and Statistics, Beijing Institute of Technology, Beijing 100081, China;
Faculty of Computational Mathematics and Cybernetics, Shenzhen MSU-BIT University, 518172 Shenzhen, China.\\ Email: xczhang.math@bit.edu.cn
}

\thanks{
    This work is supported by the National Key R\&D Program of China (No. 2023YFA1010103) and the NNSF of China (No. 12471140, 12401182, 12595282).
}

\begin{abstract}
We develop a unified PDE-probabilistic framework for pointwise gradient and Hessian estimates of Markov semigroups associated with stochastic differential equations with singular and unbounded coefficients. Under mild local structural assumptions on the diffusion matrix and integrability/regularity conditions on the drift, we obtain quantitative sharp short-time regularization estimates as well as long-time decay bounds (including exponential and polynomial rates) for the first and second spatial derivatives of the semigroup. A distinctive feature of our results is the explicit dependence of these estimates on local norms of the coefficients (through scale-invariant quantities), without requiring any global smoothness, boundedness or uniform ellipticity. In particular, our approach allows for degenerate or highly irregular behavior at infinity, subject to suitable local ellipticity and Lyapunov/ergodicity controls. As applications, we establish solvability and regularity results for Poisson equations on the whole space with singular coefficients, and we derive pointwise gradient estimates for SDEs with distributional drifts via a Zvonkin-type transform.

\bigskip


\end{abstract}

\maketitle

\tableofcontents

\section{Introduction}

\subsection{Background and motivation}
Derivative (or gradient) estimates for diffusion semigroups are fundamental  in stochastic analysis and partial differential equation (PDE) theory. They provide \emph{pointwise} control of the spatial derivatives of $\mathcal T_t\varphi$ in terms of $\varphi$ itself, quantify the instantaneous regularization of the underlying stochastic dynamics, and, from the PDE viewpoint, correspond to a priori estimates for solutions of Kolmogorov equations. Beyond regularization, gradient estimates have far-reaching implications: they typically imply the strong Feller property; through their deep connection with the Bakry-\'Emery curvature-dimension condition, they lead to Poincar\'e and log-Sobolev inequalities, which in turn govern the long-time behavior, convergence to equilibrium, and ergodic properties of the underlying stochastic system; see \cite{BGL14} and the references therein.

\smallskip
Consider the following time-homogeneous stochastic differential equation (SDE) on $\mathbb R^d$:
\begin{equation}\label{eq:sde}
\mathrm{d} X_t = b(X_t)\,\mathrm{d} t + \sqrt{2}\,\sigma(X_t)\,\mathrm{d} W_t,
\qquad X_0 = x \in \mathbb R^d,
\end{equation}
where $W$ is a standard $d$-dimensional Brownian motion. Whenever \eqref{eq:sde} admits a weak solution $X_t(x)$ for each $x$, the associated Markov semigroup is defined by
\begin{equation}\label{Semi}
\mathcal T_t \varphi(x) := \mathbb E\big[\varphi(X_t(x))\big],
\quad t\geq 0,\ x\in\mathbb R^d.
\end{equation}
A basic problem is to determine whether one has pointwise bounds of the form
\begin{equation}\label{Gr11}
|\nabla_x \mathcal T_t\varphi(x)| \leq C(t,x)\,\|\varphi\|_\infty,\quad t>0,\ x\in\mR^d,
\end{equation}
and, more importantly for applications, to understand \emph{quantitatively} how $C(t,x)$ depends on the local behavior of the coefficients. For instance, if $\sigma=\mI_d$ is the identity matrix and $|b(x)|\leq \kappa(1+|x|)$ for some $\kappa>0$, 
is it possible to establish the gradient estimate:
$$
|\nabla_x \mathcal T_t\varphi(x)| \leq C(1+|x|+t^{-1/2})\|\varphi\|_\infty,\quad t\in(0,1),\ x\in\mR^d?
$$ 
Several  influential approaches have been developed to establish such estimates.

\vspace{1mm}\noindent
\textbf{(1) Malliavin calculus and Bismut formula.}
Under global non-degeneracy and sufficient smoothness assumptions on the coefficients, Malliavin calculus leads to Bismut-type derivative formulas  expressing $\nabla \cT_t\varphi$ in terms of $\varphi(X_t)$ itself rather than derivatives of $\varphi$. More precisely, one has
\begin{align}\label{Bis1}
\nabla_x \cT_t\varphi(x) = \frac{1}{t} \mathbb{E}\left[ \varphi(X_t(x)) \int^t_0 (\sigma^{-1}(X_s(x)) J_s(x) )^*\mathrm{d} W_s \right],
\end{align}
where the asterisk denotes the transpose of a matrix, and $J_t(x) := \nabla_x X_t(x) = J_t$ is the Jacobian flow satisfying
$$
J_t = I_d + \int^t_0 \nabla b(X_s) J_s \mathrm{d} s + \int^t_0 \nabla \sigma(X_s) J_s \mathrm{d} W_s,
$$
and $I_d$ is the $d\times d$ identity matrix. This idea goes back to Bismut \cite{Bi84} and was developed systematically in the works of Elworthy--Li and many others; see \cite{EL94,Th97} and the references therein. Such formulas yield short-time bounds of order $t^{-1/2}$ and provide a versatile probabilistic representation for derivatives of solutions to linear and nonlinear PDEs. Using this formula, Thalmaier and Wang \cite[Theorem 6.1]{TW98} established gradient estimate \eqref{Gr11} with explicit dependence of $C(t,x)$ on the local $L^\infty$-norm of $b$ and the local Ricci curvature bound around $x$. Moreover, under the global assumption $\nabla\sigma, b\in L^p(\mathbb{R}^d)$ with $p>d$, a Bismut formula like \eqref{Bis1} was established in \cite{XZ20}. Recently, Wang and Zhao \cite{WZ25} also derived gradient estimates for killed SDEs with irregular drift.

\smallskip\noindent
\textbf{(2) Coupling methods.}
This probabilistic approach estimates the finite difference
$$
|\cT_t\varphi(x) - \cT_t\varphi(y)|/|x-y|
$$
by constructing a coupling $(X_t,Y_t)$ whose marginals solve the original diffusion started from $x$ and $y$, and which coalesce after a random coupling time.  It bypasses the need for differentiable coefficients, making it powerful for low-regularity settings. The seminal work of Priola and Wang \cite{PW06} established a rigorous framework for coefficients that are unbounded and merely H\"older continuous. The key assumptions are the uniform ellipticity of $\sigma(x)$ and a structural assumption like strict monotonicity:
$$
\sup_{|x-y|=r} r^{-1}\big[\|\sigma(x)-\sigma(y)\|^2 + \langle b(x)-b(y), x-y\rangle\big] \leq g(r),
$$
where $g$ is integrable near zero. This assumption is strictly weaker than   global Lipschitz continuity. Under this framework, sharp gradient estimates of the form $\|\nabla \cT_t \varphi\|_\infty \leq c_t \|\varphi\|_\infty$ are obtained. Notably, when $\int_0^\infty g(s)ds < \infty$, one recovers the characteristic short-time bound $C/\sqrt{t}$. In \cite{PW06}, these results also extend to   Dirichlet problems and yield new Liouville-type theorems.

\smallskip\noindent
\textbf{(3) Analytic/PDE methods.}
The Bernstein method is a classical analytic technique used to derive gradient estimates for solutions of PDEs, particularly in the context of parabolic problems with unbounded coefficients. The core idea is to construct an auxiliary function that combines the solution and its gradient, and then apply the maximum principle to this auxiliary function. Formally, $u(t,x):=\cT_t\varphi(x)$ solves the following Kolmogorov equation
\begin{align*}
\p_t u(t,x)=\sL u(t,x),
\end{align*}
where for $a(x)=\sigma(x)\sigma(x)^*$,
\begin{align}\label{L0}
	\sL u(x) :=\tr\big(a(x)\cdot\nabla_x^2u(x)\big)+b(x)\cdot\nabla_xu(x).
	\end{align}
Using Bernstein's technique, gradient estimates for $\cT_t\varphi$ with unbounded coefficients have been studied in \cite{Be-Lo,Ce2,Ku-Lo-Lu,Lu}, see also \cite[Chapter 1]{Ce}. However,  higher-order regularity of the coefficients are generally  assumed in these works as their methods require differentiating the equations directly.

\vspace{1mm}
Despite the substantial literature, most existing works rely on global assumptions such as uniform ellipticity, global smoothness, or boundedness of coefficients, thereby excluding a wide range of physically and biologically relevant models where coefficients display localized singularities or superlinear growth at infinity. In these regimes, classical tools become much harder to apply, and there is a pressing need for {\bf localized and quantitative} derivative estimates whose constants depend explicitly on {\bf local behavior} of the coefficients. At the same time, while short-time gradient bounds with $t^{-1/2}$ scaling are by now well understood in several frameworks, {\bf long-time derivative estimates} are considerably less developed, especially for systems involving unbounded or singular coefficients. Under global smoothness and dissipativity assumptions, long-time gradient bounds have been obtained relatively recently, see e.g., \cite{Cr-Do-Go-Ot-So,Cr-Do-Ot,Cr-Ot}, and such estimates are shown to play an essential structural role in uniform-in-time convergence results in various settings, such as Euler schemes and averaging limits \cite{Cr-Do-Go-Ot-So,Cr-Do-Ot,LWX}. For SDEs with unbounded, locally regular, and possibly singular coefficients, the lack of global smoothness obstructs Bismut-type formulas, while the absence of strong confinement limits classical coupling and spectral methods.  Consequently, ,establishing long-time gradient estimates under merely local regularity remains a significant challenge.

\subsection{Main results}

In this paper, we develop a unified quantitative framework that yields both {\bf sharp short-time} regularization bounds and {\bf long-time} decay estimates for the gradient and Hessian of the semigroup associated with \eqref{eq:sde}, allowing for locally singular and unbounded coefficients. A key novelty is that the constants in our estimates depend {\bf explicitly}, and only, on {\bf local} norms of the coefficients and local ellipticity ratios, without imposing uniform ellipticity, global smoothness, or boundedness.

\vspace{1mm}

To formulate the assumptions and main results, for a function $\rho:\mathbb{R}^d\to[1,\infty)$ we introduce the weighted space
\begin{align}\label{WF}
\mathcal{B}_{\rho}:=\mathcal{B}_{\rho}(\mathbb{R}^d):=
\Bigl\{\varphi:\ \|\varphi\|_{\mathcal{B}_{\rho}}:=\sup_{x\in\mathbb{R}^d}\frac{|\varphi(x)|}{\rho(x)}<+\infty\Bigr\}.
\end{align}
In particular, by \eqref{Mom1}, for every $\varphi\in \mathcal{B}_{\rho_0}$, the semigroup \eqref{Semi} is well defined. For $r>0$ and $x\in\mathbb{R}^d$, we write $B_r(x)$ for the ball in $\mathbb{R}^d$ centered at $x$ with radius $r$. For $p\in[1,\infty]$ and $\alpha\geq0$, we denote by $L^p(B_r(x))$ and $\mathbf{C}^\alpha(B_r(x))$ the usual $L^p$ and H\"older spaces over $B_r(x)$, and similarly use $L^p_{\mathrm{loc}}(\mathbb{R}^d)$ and $\mathbf{C}^\alpha_{\mathrm{loc}}(\mathbb{R}^d)$ for local spaces; see Section~2 for precise definitions.

\vspace{1mm}

We work under the following local assumptions and a non-explosion/moment condition:

\vspace{1mm}
\begin{enumerate}[{\bf (H$_b^\sigma$)}]
\item Suppose that $\sigma\in \mathbf{C}_{\mathrm{loc}}^{\alpha}(\mathbb{R}^{d})$ for some $\alpha\in(0,1]$,
and $b\in L^{p_b}_{\mathrm{loc}}(\mathbb{R}^d)$ for some $p_b\in(d,\infty]$,
and there are two functions $0<\lambda(x)\leq\Lambda(x)<\infty$ such that
\begin{align}\label{Ellp1}
\lambda(x)|\xi|^{2} \leq |\sigma(x) \xi|^2 \leq \Lambda(x) |\xi|^{2},\quad \forall x,\xi\in \mathbb{R}^{d}.
\end{align}
Moreover, for each $x\in\mathbb{R}^d$, there is a unique global solution $(X_t(x))_{t\geq 0}$
to SDE \eqref{eq:sde},
and there exist two functions $\rho_0,\rho_1:\mathbb{R}^d\to [1,\infty)$
and an increasing continuous function $\ell_0:[0,\infty)\to[1,\infty)$ such that
\begin{align}\label{Mom1}
\mathbb{E}\,\rho_0(X_t(x))\leq \ell_0(t)\,\rho_1(x),\quad \forall t\geq 0,\ x\in\mathbb{R}^d.
\end{align}
\end{enumerate}

\begin{remark}
Under the above regularity assumptions on $\sigma$ and $b$, together with the local ellipticity condition \eqref{Ellp1}, local well-posedness of \eqref{eq:sde} is standard; see, e.g., \cite{XZ20}. If $\rho_0=\rho_1\equiv1$, then \eqref{Mom1} holds trivially. More generally, if $\rho_0$ is a Lyapunov-type function (i.e., $\rho_0(x)\to\infty$ as $|x|\to\infty$), then \eqref{Mom1} implies non-explosion of solution.
\end{remark}

Our first main result is the following estimates, which are sharp in the short-time scaling and whose constants depend explicitly, and locally, on the coefficients.

\begin{theorem}\label{thm:main-intro}
Suppose that {\bf (H$_b^\sigma$)} holds, and $\varphi\in \mathcal{B}_{\rho_0}$.
\begin{enumerate}[(i)]
\item \textbf{(Gradient estimates)}
For any $\eps>0$, there exists a constant $C=C(\eps, d, p_b,\alpha)>0$
	such that for all $(t,x)\in (0,\infty)\times \mathbb{R}^{d}$,
	\begin{align*}
		|\nabla_x \cT_t \varphi(x) | \lesssim_{C}   \frac{\Lambda_1^{\eps}(x)\Gamma_{t}(x)}{\sqrt{t\wedge 1}}
		\left(\frac{\Lambda_1(x)}{\lambda(x)}\right)^{3\eps+d/p_b}  \ell_0(2t)\,\|\rho_1\|_{L^\infty(B_1(x))}\,\|\varphi\|_{\cB_{\rho_0}},
	\end{align*}
 where $\Lambda_{1}(x):=\Lambda(x)+1$, and for $a(x)=\sigma(x)\sigma(x)^*$ and $\theta_b=1-d/p_b$,
 \begin{align}\label{s1}
 \Gamma_{t}(x):=\sqrt{t\wedge 1} \left[\frac{ \Lambda_{1}(x) [a]^{1/\alpha}_{\bC^\alpha(B_1(x))}}{\lambda^{1+1/\alpha}(x)}+\frac{\|b\|^{1/\theta_b}_{L^{p_b}(B_{1}(x))}}{\lambda(x)^{1/\theta_b}} \right]+\frac{\Lambda_{1}(x)}{\lambda(x)}.
 \end{align}

\item \textbf{(Hessian estimates)}
If, in addition, $b\in \mathbf{C}^\alpha_{\mathrm{loc}}(\mathbb{R}^d)$ with the same $\alpha$ as in {\bf (H$_b^\sigma$)}, 
and for some $c_0\in(0,1)$,
$$
c_0\lambda(x)\leq \inf_{y\in B_1(x)}\lambda(y)\leq\sup_{y\in B_1(x)}\Lambda(y)\leq c_0^{-1}\Lambda(x),
$$
then for $\beta\in\{0,\alpha\}$,
there exists a constant $C=C(c_0, d,\alpha)>0$ such that for all $(t,x)\in  (0,\infty)\times \mathbb{R}^{d}$,
	  \begin{align*}
	 	[\nabla^2_x\mathcal{T}_t \varphi]_{\bC^\beta(B_{1/2}(x))}
	 	\lesssim_C \left(\frac{\widetilde\Gamma_t(x)}{\sqrt{t\wedge 1}}\right)^{2+\beta}
	 	\ell_0(2t)\,\|\rho_1\|_{L^\infty(B_1(x))}\,\|\varphi\|_{\cB_{\rho_0}},
	 \end{align*}
	where  \begin{align}\label{s2}
	 \widetilde{ \Gamma}_{t}(x):=\sqrt{t\wedge 1} \left[\frac{ \Lambda_{1}(x) [a]^{1/\alpha}_{\bC^\alpha(B_1(x))}}{\lambda^{1+1/\alpha}(x)}+\frac{\|b\|_{\bC^\alpha(B_{1}(x))}}{\lambda(x)} \right]+\frac{\Lambda_{1}(x)}{\lambda(x)}.
	 \end{align}
\end{enumerate}
\end{theorem}

\begin{remark}
A noteworthy feature of the bounds in \Cref{thm:main-intro} is their \emph{local} dependence on the coefficients: the right-hand sides involve only local norms on $B_1(x)$, the local ellipticity ratio $\Lambda(x)/\lambda(x)$ and the lower bound $\lambda(x)$. Of course, the semigroup value $\mathcal T_t\varphi(x)=\mathbb E[\varphi(X_t(x))]$ depends globally on the dynamics, but its \emph{spatial derivatives at a point} can still be controlled by local information. Moreover, \Cref{thm:main-intro} implies the sharp short-time asymptotics: for any $\eps>0$,
\[
\limsup_{t\to 0}\sqrt t\,\big|\nabla_x\mathcal T_t\varphi(x)\big|
\lesssim_C \Lambda_1^{\eps}(x)\left(\frac{\Lambda_1(x)}{\lambda(x)}\right)^{1+3\eps+d/p_b}\|\varphi\|_\infty,
\]
and
\[
\limsup_{t\to 0} t\,\big|\nabla_x^2\mathcal T_t\varphi(x)\big|
\lesssim_C \left(\frac{\Lambda_1(x)}{\lambda(x)}\right)^2\|\varphi\|_\infty.
\]
In particular, these leading-order short-time constants do not involve any norms of $a$ or $b$.
\end{remark}
\begin{remark}
A special case of part (i) above is that if $a(x)=\lambda I_d$ and $b\in L^{p_b}_{\text{loc}}(\mathbb{R}^{d})$ with $p_b\in(d,\infty]$, then there is a constant $C=C(d,p_b,\lambda)>0$ such that for all $(t,x)\in  (0,\infty)\times \mathbb{R}^{d}$,
\begin{align*}
		|\nabla_x \cT_t \varphi(x) | \lesssim_{C}  \ell_0(2t)\,\left( \frac{1}{\sqrt{t\wedge 1}}
	+\|b\|_{L^{p_b}(B_{1}(x))}^{1/\theta_b}\right)\|\rho_1\|_{L^\infty(B_1(x))}\,\|\varphi\|_{\cB_{\rho_0}}.
	\end{align*}
In particular, if, in addition, $b$ has polynomial growth, i.e., $|b(x)|\lesssim_C 1+|x|^m$, then
$$
|\nabla_x \cT_t \varphi(x)|\lesssim_C\ell_0(2t)\,\left( \frac{1}{\sqrt{t\wedge 1}}
+|x|^m\right)\|\varphi\|_\infty,\ \ (t,x)\in  (0,\infty)\times \mathbb{R}^{d}.
$$
\end{remark}


Although the above estimates hold for all $t\geq 0$, the function $\ell_0$ does not usually decay as $t\to\infty$. To establish decay estimates, we impose an additional, natural ergodicity assumption:

\vspace{1mm}
\begin{enumerate}[\textbf{(H$^\sigma_b$)$'$}]
\item In addition to {\bf (H$^\sigma_b$)}, there exists a unique invariant probability measure $\mu\in\mathcal{P}(\mathbb{R}^d)$ for $\cT_t$ satisfying $\int_{\mathbb{R}^d}\rho_0(x)\mu(\dif x)<\infty$, and a decreasing function $\ell_1: \mathbb{R}_+\to\mathbb{R}_+$ with $\lim_{t\to\infty}\ell_1(t)=0$ such that for all $\varphi\in\mathcal{B}_{\rho_0}$,
	\begin{align}\label{Erg1}
		|\mathcal{T}_t\varphi(x)-\mu(\varphi)|\leq \ell_1(t)\,\rho_1(x)\,\|\varphi\|_{\mathcal{B}_{\rho_0}},\ (t,x)\in(1,\infty)\times\mathbb{R}^d,
	\end{align}
	where $\rho_0,\rho_1$ are the same weights as in (\ref{Mom1}).
\end{enumerate}

The next theorem shows that the gradient and Hessian of the semigroup decay at the same rate as $\ell_1(t)$, thereby quantitatively connecting short-time regularization with long-time stabilization.

\begin{theorem}\label{thm:main-intro2}
Suppose that {\bf (H$_b^\sigma$)$'$} holds, and $\varphi\in \mathcal{B}_{\rho_0}$.
\begin{enumerate}[(i)]
\item \textbf{(Long-time gradient estimates)}
For any $\eps>0$, there exists a constant $C=C(\eps, d, p_b,\alpha)>0$
such that for all $(t,x)\in (2,\infty)\times \mathbb{R}^{d}$,
\begin{align}\label{s3}
\left|\nabla_x \mathcal{T}_t\varphi(x)\right|
\lesssim_{C}\,
\ell_1(t-1)\,
\Lambda_1^{\eps}(x)\left(\frac{\Lambda_1(x)}{\lambda(x)}\right)^{3\eps+d/p_b}
\Gamma_1(x)\,\|\rho_1\|_{L^\infty(B_{1}(x))}\,\|\varphi\|_{\mathcal{B}_{\rho_0}},
\end{align}
where $\ell_1$ comes from \eqref{Erg1} and $\Gamma_1(x)$ is given by \eqref{s1} with $t=1$.

\item \textbf{(Long-time Hessian estimates)}
If, in addition, we assume the same conditions as in Theorem \ref{thm:main-intro} (ii), then for $\beta \in \{0, \alpha\}$ there exists a constant $C = C(c_0, d,\alpha) > 0$ such that for all $(t,x) \in [2,\infty) \times \mathbb{R}^d$,
\begin{align}\label{s4}
[\nabla^2_x\mathcal{T}_t \varphi]_{\mathbf{C}^\beta(B_{1/2}(x))}
\lesssim_C
\ell_1(t-1)\,\widetilde{\Gamma}_1^{2+\beta}(x)\,
\|\rho_1\|_{L^\infty(B_{1}(x))}\,\|\varphi\|_{\mathcal{B}_{\rho_0}},
\end{align}
where $\widetilde{\Gamma}_1(x)$ is given by \eqref{s2} with $t=1$.
\end{enumerate}
\end{theorem}

Our method to prove the above estimates is quite unified, and both the regularity assumptions on the coefficients and the dissipative condition for the long-time estimates are much weaker than existing results in the literature. In fact, we work at the level of the more general time-inhomogeneous Kolmogorov equation
\begin{align}\label{pde311}
\partial_{t}u(t,x)+\sL_t u(t,x) +f(t,x)=0,
\end{align}
with
\begin{align*}
\sL_t u(x) :=\tr\big(a(t,x)\cdot\nabla_x^2 u(x)\big)+b(t,x)\cdot\nabla_xu(x).
\end{align*}
The backbone of our argument is to establish pointwise control of $|\nabla_xu(t,x)|$ and $|\nabla^2_xu(t,x)|$ for solutions to \eqref{pde311} in terms of $u(t,x)$ itself; see Theorems~\ref{grad1} and \ref{grad3} below. This is achieved by deriving new a priori local $L^q_tL^p_x$ maximal regularity and local Schauder estimates for \eqref{pde311} with unbounded coefficients, where the dependence on coefficients is made explicit via scaling and interpolation, and then applying embedding theorems on domains with a novel scale selection procedure that determines the intrinsic radius based on local ellipticity and regularity norms. The short-time estimates in Theorem~\ref{thm:main-intro} follow via probabilistic representations of solutions. As a consequence, we also obtain short-time gradient estimates for time-inhomogeneous SDEs; see Theorems~\ref{Cor43} and \ref{Cor44} below. The long-time estimates in Theorem~\ref{thm:main-intro2} follow from the quantitative short-time bounds, the semigroup property, and the ergodic decay assumption, thereby linking short-time regularization and long-time stabilization.

\vspace{1mm}

Below, we provide two examples illustrating assumption {\bf (H$^\sigma_b$)$'$}, in particular the ergodic bound \eqref{Erg1}.

\begin{example} \rm \textbf{(Lyapunov condition and exponential ergodicity)}
Suppose that there exists a $C^2$-function $V:\mathbb{R}^{d}\to [1,\infty)$ and constants $c_0,c_1\in\mathbb{R}$ such that $\lim_{|x|\to\infty}V(x)=\infty$ and
\begin{align}\label{GDC}
\sL V(x)\leq c_0 V(x)+c_1,\qquad \forall x \in \mathbb{R}^{d},
\end{align}
where $\sL$ is defined in \eqref{L0}. Then \eqref{Mom1} holds with $\rho_0=\rho_1=V$ (see Theorem 2.1 in \cite{Me-Tw}) and $\ell_0(t)=\mathrm{e}^{c_0 t}$. If, moreover, $c_0<0$ in \eqref{GDC}, then \eqref{Erg1} holds with $\rho_0(x)=V(x)$, $\rho_1(x)=cV(x)$ and $\ell_1(t)=\mathrm{e}^{-\gamma t}$ for some constants $c,\gamma>0$ (see Theorem 6.1 in \cite{Me-Tw}). In particular, if $a=\sigma\sigma^*$ and $b$ satisfy one of the following dissipativity conditions, then one can choose $V$ so that \eqref{GDC} holds with $c_0<0$ (see, e.g., \cite{Eb-Gu-Zi}):

\begin{enumerate}
\item[{\rm (a)}] For any $\gamma>0$, there exist $\vartheta,c>0$ and $q\geq2$ such that
\begin{align*}
\gamma|a(x)|+\<b(x),x\>\leq -\vartheta |x|^{q}+c.
\end{align*}
Then \eqref{GDC} holds for some $V\in\mathbf{C}^2$ with $V(x):=1+|x|^p$ with $p>0$  outside a ball $B_R(0)$ and any  $c_0<0$ if $q>2$ (with a finite constant $c_{1}(p,c_0)$), or if $q=2$ with $c_0\in(-p\vartheta,0)$.

\item[{\rm (b)}] For any $\gamma>0$, there exist $\vartheta,c>0$ and $q\geq1$ such that
\begin{align}\label{xq}
\gamma q|a(x)||x|^q+\<b(x),x\>\leq -\vartheta |x|^{q}+c.
\end{align}
Then \eqref{GDC} holds for some $V\in\mathbf{C}^2$ with $V(x)=\exp(\gamma |x|^q)$ outside a ball $B_R(0)$ 
and any $c_0<0$ (with a finite constant $c_{1}(\gamma,c_0)$), provided $q>1$, or $q=1$ and $\vartheta>-\frac{c_0}{\gamma}$.

\item[{\rm (c)}]
Under \eqref{xq}, \eqref{GDC} also holds for some $V\in\mathbf{C}^2$ with $V(x)=\exp(\gamma |x|^p)$ with $p\in[1,q)$ outside a compact set and any $c_0<0$ (with a finite constant $c_{1}(\gamma,c_0)$).
\end{enumerate}
\end{example}

\begin{example} \rm \textbf{(Geometric drift condition and polynomial ergodicity)}
Assume that $\sigma(x)$ is uniformly elliptic, i.e., there exists $0<\lambda<1$ such that
\[
\lambda|\xi|^{2} \leq |\sigma(x) \xi|^2 \leq \lambda^{-1}|\xi|^{2}, \quad \forall x,\xi \in \mathbb{R}^{d},
\]
and the drift $b$ satisfies the geometric drift condition
\begin{align*}
\lim_{|x|\rightarrow \infty}\<x,b(x)\> =-\infty.
\end{align*}
Then {\bf (H$^\sigma_b$)$'$} holds with $\rho_0(x)=1+|x|^{m}$ for any $m>0$, $\ell_1(t)=t^{-\gamma}$ for any $\gamma>0$, and $\rho_1(x)=c(1+|x|^{m+\beta})$, where $c>0$ and $\beta>2\gamma$ are constants.

Indeed, under the above hypotheses, \cite[Proposition 1]{PV03} implies that for any $m>0$ and $t\geq0$,
\[
\mathbb{E}|X_t(x)|^m\leq C_d(1+|x|^m).
\]
Moreover, there exists a unique invariant measure $\mu$ for \eqref{eq:sde} which is polynomially ergodic in total variation, i.e., for any $t>0$, $k>0$, and $n>2k+2$, there exists $C$ such that
\[
\|\mathbb{P}\circ X_t(x)^{-1}-\mu\|_{\rm var}\lesssim_{C_d} (1+|x|^{n})(1+t)^{-(k+1)},
\]
where $\|\cdot\|_{\rm var}$ denotes the total variation norm of a signed measure. Taking $\rho_0(x)=1+|x|^m$, for any $\gamma>0$ and $\beta>2\gamma$, we can choose $k,n$ and $p,q>1$ with $1/p+1/q=1$ such that $\gamma=(k+1)/q$ and $\beta=n/q$, and deduce that for any $\varphi\in\mathcal{B}_{\rho_0}$,
\begin{align*}
|\mathcal{T}_t\varphi(x)-\mu(\varphi)|
&=\left|\int_{\mathbb{R}^{d}} \varphi(y)\big[\mathbb{P}\circ X_t(x)^{-1}(\mathrm{d} y)-\mu(\mathrm{d} y)\big]\right|\\
&\leq \left(\int_{\mathbb{R}^{d}} |\varphi(y)|^p \big[\mathbb{P}\circ X_t(x)^{-1}(\mathrm{d} y)+\mu(\mathrm{d} y)\big]\right)^{1/p}\,
\|\mathbb{P}\circ X_t(x)^{-1}-\mu\|_{\rm var}^{1/q}\\
&\lesssim_{C_{d}} \|\varphi\|_{\mathcal{B}_{\rho_0}} (1+|x|^{mp})^{1/p} (1+|x|^{n})^{1/q}(1+t)^{-\gamma}\\
&\lesssim_{C_{d}}\|\varphi\|_{\mathcal{B}_{\rho_0}} (1+|x|^{m+\beta})(1+t)^{-\gamma},
\end{align*}
which yields \eqref{Erg1}.
\end{example}

\subsection{Applications and future perspectives}

The explicit and local nature of our estimates is particularly useful in inhomogeneous, unbounded or singular environments: it enables robust control under truncation, scaling, and localization, and is therefore well-suited to problems in homogenization and averaging, stability analysis, and numerical approximation of singular stochastic systems. Below, we first highlight two concrete applications and then outline several methodological directions enabled by our framework.

\subsubsection{Poisson equation in the whole space.}
The long-time derivative estimates allow us to study the Poisson/corrector equation on $\mathbb{R}^{d}$:
\begin{align}\label{000}
-\sL u(x)=f(x),\quad x\in\mathbb{R}^{d},
\end{align}
where $\sL$ is given by \eqref{L0} and is the generator of \eqref{eq:sde}. Equations of the form \eqref{000} are also referred to as the cell problems in periodic homogenization, and are known to be a key tool in stochastic averaging, diffusion approximation, moderate and large deviations, and numerical approximation of invariant measures for SDEs; see, e.g., \cite{MST,PP,PV01,PV03,PavStu08}. The well-posedness and regularity of $u$ were studied in \cite{PV01,PV03}; see also \cite{Ro-Xi1} and references therein. However, these works typically impose boundedness-type assumptions, which restrict applicability. More recently, under dissipativity assumptions, \cite{Cr-Do-Go-Ot-So} studied \eqref{000} with unbounded coefficients, but still required uniform ellipticity and higher-order regularity of $\sigma$, $b$, and the forcing $f$.

We establish well-posedness of \eqref{000} in local Sobolev spaces for singular and unbounded coefficients, and provide quantitative gradient and Schauder-type bounds for the solution $u$ with explicit dependence on local coefficient norms and local ellipticity ratios; see Theorems~5.1 and~5.2 below. These quantitative estimates are crucial in proving uniform-in-time convergence of multi-scale problems: in averaging/homogenization and diffusion approximation, effective coefficients and error terms are typically expressed through correctors solving \eqref{000} (see \cite{XZ26}).

\subsubsection{SDEs with distributional drifts.}
As an illustration of the scope of our localized framework, we apply the main results to derive gradient estimates for SDEs with distributional drift:
\[
\dif X_t=(b_1+b_2)(X_t)\dif t+\dif W_t,
\]
where $b_1$ belongs to a negative-order Sobolev/Bessel potential class and $b_2$ has at most linear growth. Such models arise naturally in fluid dynamics and rough path theory, and classical approaches to gradient estimates break down due to the lack of pointwise meaning of the drift (see \cite{CC18, GP24, HZ25} and references therein). Using a Zvonkin-type transform, we remove the singular component $b_1$ and reduce the equation to an SDE with H\"older diffusion and unbounded drift, to which our theorems apply. This yields, to the best of our knowledge, the first pointwise gradient estimate for semigroups associated with SDEs whose drift lies in negative-order Sobolev spaces; see Theorem~6.3 below.

\subsubsection{Methodological perspectives}
Beyond the specific applications presented above, our framework provides a foundation for several research directions:

\begin{itemize}
\item \textbf{Nonlinear PDEs via linearization.}
Many nonlinear (quasilinear) parabolic equations reduce locally to Kolmogorov-type operators after linearization. Quantitative a priori bounds with explicit dependence on local coefficient norms provide natural input for bootstrap and continuation arguments. This perspective is closely related to the classical theory of linear and quasilinear parabolic equations \cite{LSU68,Lieberman96} and to modern $L^p$-based regularity theories for rough coefficients \cite{KrylovVMO06,Kr08}.

\item \textbf{Homogenization, averaging, and diffusion approximation.}
Quantitative weak convergence often reduces to controlling correctors and remainder terms expressed through Poisson equations and semigroup expansions. Our localized short-time derivative bounds enable precise error control in multi-scale expansions even when global ellipticity or regularity fails; see \cite{Khas68,PV01,PV03,PavStu08} for representative frameworks where Poisson equations are structural. Moreover, the precise long-time bounds fit naturally into error decompositions where the time horizon must be controlled uniformly, see \cite{XZ26}.

\item \textbf{Quantitative numerical analysis.}
Weak error analysis for discretization schemes requires derivative estimates for backward Kolmogorov equations. Our results naturally feed into local weak error bounds under unbounded coefficients via Lyapunov weights, and into long-time error estimates (such as approximation of invariant measures) when combined with ergodic decay. For classical weak error expansions and numerical perspectives, see \cite{MST,TalayTubaro90,Platen99} and references therein.

\item \textbf{Sensitivity and robustness.}
Pointwise gradient bounds quantify the stability of expectations under perturbations, which is crucial for parameter sensitivity and robustness of Markov semigroups under coefficient perturbations. Because our constants are local, stability can be proved by localization and patching arguments, a viewpoint compatible with Harris-type ergodic frameworks and perturbation arguments; see e.g. \cite{Eb-Gu-Zi,Me-Tw}.

\item \textbf{Stochastic control with rough coefficients.}
In stochastic control, verification and regularity of feedback controls depend on spatial derivatives of the value function (solutions to HJB equations). Localized semigroup regularization estimates provide building blocks for mollified verification and stability of near-optimal controls when coefficients are unbounded or rough; see, e.g., \cite{FlemingSoner06} for the viscosity-solution approach.
\end{itemize}

The rest of the paper is organized as follows. Section~\ref{sec:prelim} introduces notation and analytic tools. Section~\ref{sec:grad} develops local PDE estimates for Kolmogorov equations. Section~4 derives pointwise gradient and Hessian estimates for the SDE semigroup. Sections~5 and~6 present the applications to Poisson equations and SDEs with distributional drifts, respectively.

\vspace{1mm}
Throughout the paper, $C$ denotes a positive constant whose value may change from line to line. We write $A\lesssim B$ to mean $A\leq C B$ for some inessential constant $C>0$, and we will specify the dependence of constants when necessary.

\section{Preliminaries}\label{sec:prelim}

In this section, we introduce the local Sobolev and H\"older spaces, together with several interpolation inequalities that will be used repeatedly below.

\subsection{Local Sobolev spaces and anisotropic interpolation}

Let $Q=I\times\mathcal{O}\subset\mathbb{R}\times\mathbb{R}^{d}$, where $I\subset\mathbb{R}$ is an open interval and $\mathcal{O}\subset\mathbb{R}^d$ is an open domain. For $q,p\in[1,\infty]$, we define the anisotropic Lebesgue norm
\[
\|f\|_{\mathbb{L}^q_p(Q)}:=\|f\|_{L^q(I;L^p(\mathcal{O}))}
:=\left(\int_I\left(\int_{\mathcal{O}} |f(t,x)|^p\,\dif x\right)^{q/p}\dif t\right)^{1/q},
\]
with the usual modifications when $q=\infty$ or $p=\infty$. When $Q=\mathbb{R}\times\mathbb{R}^d$, we simply write
\[
\|f\|_{q,p}:=\|f\|_{L^q_tL^p_x}:=\|f\|_{L^q(\mathbb{R};L^p(\mathbb{R}^d))},\qquad
\|f\|_{\infty}:=\|f\|_{L^\infty(\mathbb{R};L^\infty(\mathbb{R}^d))}.
\]
Let $\mW^{1,2}_{q,p}(Q)$ denote the anisotropic Sobolev space consisting of all measurable functions $f$ on $Q$ such that $\partial_t f$ and $\nabla_x^2 f$ exist in the weak sense and
\[
\|f\|_{\mW^{1,2}_{q,p}(Q)}:=\|f\|_{\mathbb{L}^q_p(Q)}+\|\partial_t f\|_{\mathbb{L}^q_p(Q)}+\|\nabla_x^2 f\|_{\mathbb{L}^q_p(Q)}<\infty.
\]
We also introduce the local Sobolev space
\[
\mW^{1,2}_{q,p;\mathrm{loc}}(Q):=\big\{u:\ u\phi\in \mW^{1,2}_{q,p}(Q)\ \text{for all }\phi\in \mathbf{C}^\infty_c(Q)\big\}.
\]
Similarly, $\mW^2_p(\mathcal{O})$ and $\mW^2_{p;\mathrm{loc}}(\mathcal{O})$ denote the global and local Sobolev spaces over $\mathcal{O}$, respectively.

We record the following Gagliardo--Nirenberg type anisotropic interpolation inequality.

\begin{lemma}\label{inter1}
Let $p,q\in(1,\infty)$ satisfy
$
\frac{d}{p}+\frac{2}{q}<1.
$
Let $q_0\in[q,\infty]$ and $p_0\in[p,\infty]$, and define
$
\theta:=1-\frac{d}{p}-\frac{2}{q}+\frac{d}{p_0}+\frac{2}{q_0}.
$
Then there exists a constant $C_1=C_1(d,p,q,p_0,q_0)>0$ such that for
any $u\in \mW^{1,2}_{q,p}(\mathbb{R}\times\mathbb{R}^d)$ and all $\eps_1,\eps_2>0$,
\begin{equation}\label{eq:B}
\|\nabla_x u\|_{q_0,p_0}
\lesssim_{C_1}
\eps_1\|\nabla_x^2 u\|_{q,p}
+\eps_2\|\partial_t u\|_{q,p}
+\eps_1^{(d/p_0-1-d/p)/\theta}\eps_2^{(2/q_0-2/q)/\theta}\|u\|_{q,p}.
\end{equation}
\end{lemma}

\begin{proof}
Fix $\lambda_1,\lambda_2>0$ and define the parabolic rescaling
\begin{equation*}
\widetilde u(t,x):=u(\lambda_1 t,\lambda_2 x).
\end{equation*}
Since $\frac{d}{p}+\frac{2}{q}<1$, it is known that there exists a constant $C=C(d,p,q,p_0,q_0)>0$ such that (see \cite[Lemma 10.2]{Kr-Ro})
\begin{equation}\label{eq:A}
\|\nabla_x \widetilde u\|_{q_0,p_0}
\lesssim_{C}
\|\nabla_x^2 \widetilde u\|_{q,p}+\|\partial_t \widetilde u\|_{q,p}+\|\widetilde u\|_{q,p}.
\end{equation}
A direct computation gives the scaling relations
\[
\begin{aligned}
\|\nabla_x \widetilde u\|_{q_0,p_0}
&=\lambda_2^{1-d/p_0}\lambda_1^{-1/q_0}\|\nabla_x u\|_{q_0,p_0},
\qquad
\|\widetilde u\|_{q,p}
=\lambda_2^{-d/p}\lambda_1^{-1/q}\|u\|_{q,p},\\
\|\nabla_x^2 \widetilde u\|_{q,p}
&=\lambda_2^{2-d/p}\lambda_1^{-1/q}\|\nabla_x^2 u\|_{q,p},
\qquad
\|\partial_t \widetilde u\|_{q,p}
=\lambda_2^{-d/p}\lambda_1^{1-1/q}\|\partial_t u\|_{q,p}.
\end{aligned}
\]
Substituting into \eqref{eq:A} yields
\begin{align}\label{As5}
\begin{split}
\|\nabla_x u\|_{q_0,p_0}
&\lesssim_{C}
\lambda_2^{1+d/p_0-d/p}\lambda_1^{1/q_0-1/q}\|\nabla_x^2 u\|_{q,p}\\
&\quad+\lambda_2^{d/p_0-d/p-1}\lambda_1^{1+1/q_0-1/q}\|\partial_t u\|_{q,p}\\
&\quad+\lambda_2^{d/p_0-d/p-1}\lambda_1^{1/q_0-1/q}\|u\|_{q,p}.
\end{split}
\end{align}
Now choose $\lambda_1,\lambda_2>0$ such that
\[
\lambda_2^{1+d/p_0-d/p}\lambda_1^{1/q_0-1/q}=\eps_1,
\qquad
\lambda_2^{d/p_0-d/p-1}\lambda_1^{1+1/q_0-1/q}=\eps_2.
\]
Solving these equations gives
\[
\lambda_2
=\eps_1^{(1+1/q_0-1/q)/\theta}\,\eps_2^{(1/q-1/q_0)/\theta},
\qquad
\lambda_1
=\eps_1^{(1+d/p-d/p_0)/\theta}\,\eps_2^{(1+d/p_0-d/p)/\theta},
\]
where $\theta=1-\frac{d}{p}-\frac{2}{q}+\frac{d}{p_0}+\frac{2}{q_0}$.
Consequently,
\begin{equation}\label{As7}
\lambda_2^{d/p_0-d/p-1}\lambda_1^{1/q_0-1/q}
=\eps_1^{(d/p_0-1-d/p)/\theta}\,\eps_2^{(2/q_0-2/q)/\theta}.
\end{equation}
Combining \eqref{As5} and \eqref{As7} gives the desired estimate \eqref{eq:B}.
\end{proof}

\begin{remark}
Let $u\in \mW^{1,2}_{q,p}(\mathbb{R}\times\mathbb{R}^d)$ with $p,q\in(1,\infty)$ and $\frac{d}{p}+\frac{2}{q}<1$.
For $p_0\in[p,\infty]$, the estimate
\begin{equation*}
\|\nabla_x u\|_{q,p_0}
\lesssim_{C}
\eps\|\nabla_x^2 u\|_{q,p}
+\eps^{(d/p_0-1-d/p)/\theta_0}\|u\|_{q,p},
\quad
\theta_0:=1-\tfrac{d}{p}+\tfrac{d}{p_0},
\end{equation*}
follows from the classical (spatial) Gagliardo--Nirenberg interpolation inequality applied to $u(t,\cdot)$ for a.e.\ $t$ and then integrated in time; see, e.g., \cite[Theorem 1.1]{FFRS}.
\end{remark}

\subsection{H\"older spaces in a domain}

Let $\mathcal{O}\subset\mathbb{R}^d$ be an open domain. For $\alpha\in(0,1]$ and $k\in\mathbb{N}_0$,
we use $\mathbf{C}^{k+\alpha}(\mathcal{O})$ to denote the H\"older spaces over $\mathcal{O}$ with norm
$$
\|f\|_{\mathbf{C}^{k+\alpha}(\mathcal{O})}:=\|f\|_{L^\infty(\mathcal{O})}+\cdots+\|\nabla^k f\|_{L^\infty(\mathcal{O})}+[\nabla^k f]_{\mathbf{C}^{\alpha}(\mathcal{O})},
$$
where
$$
[\nabla^k f]_{\mathbf{C}^{\alpha}(\mathcal{O})}:=\sup_{x\neq y\in\mathcal{O}}\frac{|\nabla^k f(x)-\nabla^k f(y)|}{|x-y|^\alpha}.
$$
We shall use the conventions
\[
[f]_\alpha:=[f]_{\mathbf{C}^\alpha(\mathbb{R}^d)},\qquad
\|f\|_{\alpha}:=\|f\|_{\infty}+  [f]_\alpha.
\]
The local H\"older space is defined by
\[
\mathbf{C}^{k+\alpha}_{\mathrm{loc}}(\mathcal{O})
:=\{f:\ f\phi\in \mathbf{C}^{k+\alpha}(\mathcal{O})\ \text{for all }\phi\in\mathbf{C}^\infty_c(\mathcal{O})\}.
\]

We recall the following iteration lemma (cf.\ \cite[Lemma 4.3]{Ha-Li}), which will be used for proving our local $L^p$-maximal regularity estimates and Schauder's estimates.

\begin{lemma}\label{Le26}
Let $0\leq \tau_1<\tau_2\leq 1$, $A,B>0$, and let $h:[\tau_1,\tau_2]\to[0,\infty)$ be bounded. Suppose that for some $\gamma\geq 0$ and $\vartheta\in(0,1)$,
\[
h(\tau)\leq\vartheta\, h(\tau')+(\tau'-\tau)^{-\gamma}A+B,
\qquad \tau_1\leq\tau<\tau'\leq\tau_2.
\]
Then there exists $C=C(\gamma,\vartheta)>0$ such that for all $\tau_1\leq\tau<\tau'\leq\tau_2$,
\[
h(\tau)\lesssim_C (\tau'-\tau)^{-\gamma}A+B.
\]
\end{lemma}

We also need the following interpolation inequality in H\"older spaces, which is used in local Schauder estimates.

\begin{lemma}\label{inter11}
	Let $0\leq\beta\leq\alpha\leq 1$. For any $j=0,1,2$,
	there exist constants $C=C(d,j,\beta,\alpha)>0$
	and $\gamma=\gamma(j,\alpha,\beta)>0$ such that for all $u\in \bC^{2+\alpha}(B_1)$,
	$\eps>0$ and $\frac12\leq\sigma<\tau\leq1$,
	\begin{align}\label{Ho1}
	[\nabla^j u]_{\bC^\beta(B_\sigma)}\lesssim_C \eps^{2-j+\alpha-\beta}[\nabla^2_xu]_{\bC^\alpha(B_\tau)}
	+\eps^{-j-\beta}[\eps^{2+\alpha}(\tau-\sigma)^{-\gamma}+1]\| u\|_{L^\infty(B_\tau)}.
	\end{align}
\end{lemma}

\begin{proof}
	It is well known that the following global interpolation inequality holds (cf. \cite[Theorem 3.2.1]{Kr96}):
	for some $C=C(d,j,\beta,\alpha)>0$,
	$$
	[\nabla^j u]_\beta\lesssim_C [\nabla^2_xu]_\alpha+\| u\|_\infty.
	$$
	For $\eps>0$ define
	$$
	u_\eps(x):=u(\eps x).
	$$
	Then
	$$
	[\nabla^j u_\eps]_\beta=\eps^{j+\beta}[\nabla^j u]_\beta.
	$$
	Thus
	$$
	\eps^{j+\beta}[\nabla^j u]_\beta\lesssim \eps^{2+\alpha}[\nabla^2_xu]_\alpha+\| u\|_\infty.
	$$
	Dividing both sides by $\eps^{j+\beta}$, we derive
	\begin{align}\label{DS210}
	[\nabla^j u]_\beta\lesssim_C \eps^{2-j+\alpha-\beta}[\nabla^2_xu]_\alpha+\eps^{-j-\beta}\| u\|_\infty.
	\end{align}
	Fix $\frac{1}{2}\leq\sigma<\tau\leq1$.
	Let $\eta_0(x)$ and $\eta_1(x)$ be smooth cut-off functions supported in $B_{(\tau+\sigma)/2}$
	and $B_{\tau}$, respectively, and satisfy that
	$$
	0\leq \eta_0 \leq 1 \,\,\, {\text {in}} \,\,\, B_{(\tau+\sigma)/2},\ \
	\eta_0\equiv 1\,\, {\text {in}}  \,\,B_{\sigma},
	$$
	and
	$$
	0\leq \eta_1 \leq 1 \,\,\, {\text {in}} \,\,\, B_{\tau},\ \
	\eta_1\equiv 1\,\, {\text {in}}  \,\,B_{(\tau+\sigma)/2},
	$$
	and for any $j=0,1,2$ and $\beta\in[0,1]$, there is a $C=C(j,\beta,d)>0$ such that
	\begin{align}\label{DS211}
	[\nabla^j \eta_i]_\beta\leq C(\tau-\sigma)^{-j-\beta},\ \ i=0,1.
	\end{align}
	Let $\upsilon:=(\tau+\sigma)/2$ and write
	$$
	\|\cdot\|_{\beta;\sigma}:=\|\cdot\|_{\bC^\beta(B_\sigma)},\ \ \|\cdot\|_{\infty;\sigma}:=\|\cdot\|_{L^\infty(B_\sigma)}.
	$$
	Now by \eqref{DS210}, the chain rule and \eqref{DS211}, we have for any $\eps>0$,
	\begin{align}
	[\nabla^j u]_{\beta;\sigma}&\leq[\nabla^j(u\eta_0)]_{\beta}
	\lesssim \eps^{2-j+\alpha-\beta}[\nabla^2 (u\eta_0)]_\alpha+\eps^{-j-\beta}\| u\eta_0\|_\infty\no\\
	&\lesssim \eps^{2-j+\alpha-\beta}\left([\nabla^2_xu]_{\alpha;\upsilon}+[\nabla_xu]_{\alpha;\upsilon}\|\nabla\eta_0\|_\infty+\|\nabla_xu\|_{\infty;\upsilon}[\nabla\eta_0]_\alpha\right)
	+\eps^{-j-\beta}\| u\|_{\infty;\upsilon}\no\\
	&\lesssim \eps^{2-j+\alpha-\beta}\left([\nabla^2_xu]_{\alpha;\upsilon}+(\tau-\sigma)^{-1-\alpha}\|\nabla_xu\|_{\alpha;\upsilon}\right)
	+\eps^{-j-\beta}\| u\|_{\infty;\upsilon}.\label{DS213}
	\end{align}
	In particular, if one takes $j=1, \beta=\alpha$ and replaces $\eps$ by $(\eps (\tau-\sigma))^{1+\alpha}$, then
	$$
	[\nabla_xu]_{\alpha;\sigma}\lesssim
	\eps^{1+\alpha}\left([\nabla^2_xu]_{\alpha;\tau}+\|\nabla_xu\|_{\alpha;\tau}\right)
	+(\eps(\tau-\sigma))^{-(1+\alpha)^2}\| u\|_{\infty;\tau};
	$$
	and if one takes $j=1, \beta=0$ and replaces $\eps$ by $\eps (\tau-\sigma)$, then
	$$
	\|\nabla_xu\|_{\infty;\sigma}\lesssim
	\eps^{1+\alpha}\left([\nabla^2_xu]_{\alpha;\tau}+\|\nabla_xu\|_{\alpha;\tau}\right)
	+(\eps(\tau-\sigma))^{-1}\| u\|_{\infty;\tau}.
	$$
	Combining the above two estimates and taking $\eps$ to be small enough, we obtain
	that for some $C=C(d,\alpha)>0$ and all $\frac12\leq\sigma<\tau\leq 1$,
	$$
	\|\nabla_xu\|_{\alpha;\sigma}\leq\tfrac12\|\nabla_xu\|_{\alpha;\tau}+
	[\nabla^2_xu]_{\alpha;\tau}+C(\tau-\sigma)^{-(1+\alpha)^2}\| u\|_{\infty;\tau}.
	$$
	By Lemma \ref{Le26}, we get
	$$
	\|\nabla_xu\|_{\alpha;\upsilon}\lesssim[\nabla^2_xu]_{\alpha;\tau}+(\tau-\sigma)^{-(1+\alpha)^2}\| u\|_{\infty;\tau}.
	$$
	Substituting this into \eqref{DS213}, we obtain the desired estimate.
\end{proof}
\begin{remark}
By \eqref{Ho1}, for any $0\leq\beta\leq\alpha\leq 1$, $j=0,1,2$ and $\eps>0$, we have
$$
[\nabla^j u]_{\bC^\beta(B_{1/2})}\lesssim_C \eps^{2-j+\alpha-\beta}\|\nabla^2_xu\|_{\bC^\alpha(B_1)}
+\eps^{-j-\beta}\| u\|_{L^\infty(B_1)}.
$$
\end{remark}

\section{Kolmogorov equations with unbounded coefficients}\label{sec:grad}

This section is devoted to establishing a priori local regularity estimates for solutions to time-inhomogeneous Kolmogorov equations with unbounded and possibly singular coefficients. A principal novelty is that all bounds come with {explicit} dependence on {local} norms of the coefficients, which is crucial for applications.

Let $Q:=I\times \mathcal{O}\subset\mathbb{R} \times \mathbb{R}^{d}$ be a space-time domain, where $I=(S,T)\subset\mathbb{R}$ is an open interval and $\mathcal{O}\subset\mathbb{R}^d$ is an open domain. We consider the time-inhomogeneous Kolmogorov equation in $Q$,
\begin{align}\label{pde31}
\partial_{t}u(t,x)+\sL_t u(t,x) +f(t,x)=0,
\end{align}
where
\[
\sL_t u(x)   :=\tr\big(a(t,x)\cdot\nabla_x^2 u(x)\big)+b(t,x)\cdot\nabla_xu(x).
\]
Here $a:\mathbb{R}\times\mathbb{R}^d\to\mathbb{R}^d\otimes\mathbb{R}^d$ is a symmetric matrix-valued measurable function, and $b:\mathbb{R}\times\mathbb{R}^d\to\mathbb{R}^d$ is a measurable vector field. Throughout this section we impose the following local regularity and ellipticity assumption on $a$:

\begin{enumerate}[{\bf(H$^\alpha_a$)}]
\item Suppose that $a\in L^\infty_{\text{loc}}\big(\mathbb{R};\mathbf{C}_{\text{loc}}^{\alpha}(\mathbb{R}^{d})\big)$
for some $\alpha\in(0,1]$, and there exist two functions $0<\lambda(x)\leq\Lambda(x)<\infty$ such that,
for all $t\in\mathbb{R}$ and $x,\xi\in \mathbb{R}^{d}$,
\[
\lambda(x)|\xi|^{2} \leq \langle a(t,x)\xi,\xi\rangle \leq \Lambda(x)|\xi|^{2}.
\]
\end{enumerate}

\vspace{1mm}
We first define strong solutions to \eqref{pde31}.

\begin{definition}\label{Def31}
Let $p,q\in[1,\infty)$. A function $u\in \mW^{1,2}_{q,p;{\rm loc}}(Q)$ is called a {strong solution}
of \eqref{pde31} in $Q$ if $\sL_\cdot u + f\in L^1_{\rm loc}(Q)$ and for every
$\phi\in \mathbf{C}^\infty_c(Q)$,
\[
\langle\!\langle u,\partial_t\phi\rangle\!\rangle
=
\langle\!\langle\sL_\cdot u + f,\phi\rangle\!\rangle,
\]
where
\[
\langle\!\langle f,g\rangle\!\rangle:=\int_{Q} f(t,x)g(t,x)\dif x\dif t.
\]
Equivalently, for every $g\in\mathbf{C}_c^\infty(\mathcal{O})$ and for a.e.\ $S<t<s<T$,
\begin{align}\label{Id1}
\langle u(s),g\rangle=\langle u(t),g\rangle-\int_t^s\langle\sL_r u(r) + f(r),g\rangle\dif r,
\end{align}
where
\[
\langle f,g\rangle:=\int_{\mathcal{O}} f(x)g(x)\dif x.
\]
\end{definition}

\begin{remark}
If \eqref{Id1} holds with $s=T$ (in the above sense) and $u(T)=\varphi\in L^1_{\mathrm{loc}}(\mathcal{O})$, then $u$ is called a solution to \eqref{pde31} with terminal value $u(T)=\varphi$.
\end{remark}

To state the main results, we recall the parabolic cylinder with center $(t,x)$ and radius $R>0$:
\begin{align}\label{QR1}
\cQ_{R}(t,x):=I_{R}(t)\times B_R(x):=(t-R^2,t+R^2)\times B_R(x).
\end{align}
Moreover, we also use the following notations:
$$
\|f\|_{{\mL}^q_p(Q_{R}(t,x))}:=\|f\|_{L^q(I_R(t);L^p(B_R(x)))}
$$
and
\[
[f]_{{L}^\infty\mathbf{C}^\alpha(Q_{R}(t,x))}:=\sup_{s\in I_R(t)}[f(s,\cdot)]_{\mathbf{C}^\alpha(B_R(x))}.
\]

The first main result of this section provides an a priori pointwise gradient bound for the solution with constants depending explicitly on local norms of the coefficients.

\begin{theorem}[Gradient estimate]\label{grad1}
Suppose that {\bf (H$^\alpha_a$)} holds, and $b\in L^{q_b}_{\rm loc}(\mathbb{R}; L^{p_b}_{\rm loc}(\mathbb{R}^d))$ for some $q_b,p_b\in(1,\infty]$ with
$\theta_b:=1-\frac {d}{p_b}-\frac {2}{q_b}>0$. Let
$q\in(1,\infty)\cap(1,q_b]$ and $p\in(1,\infty)\cap(1,p_b]$ satisfy $\theta:=1-\tfrac d{p}-\tfrac 2{q}>0.$
For given $S<T$, there exists a constant $C=C(d,p,q,p_b,q_b,\alpha)>0$ such that for any solution
$u \in \mW^{1,2}_{q,p;{\rm loc}}\big((S,T)\times \mathbb{R}^{d}\big)$ of PDE \eqref{pde31} and $(t, x)\in [\frac {T+S}2,T)\times \mathbb{R}^{d}$,
\begin{align}\label{grad0}
|\nabla_{x} u(t,x)|
&\lesssim_{C}\Lambda^{1/q}_1(x)\left(\frac{\Lambda_1(x)}{\lambda(x)R }\right)^{2-\theta}
\|u\|_{\mL^q_tL^p_x(Q_{R}(t,x))}
+\frac{R^{\theta}\Lambda_1^{1/q}(x)}{\lambda(x)}\|f\|_{\mL^q_tL^p_x(Q_{R}(t,x))},
\end{align}
where $\Lambda_1(x):=\Lambda(x)+1$,
$R_{0}:=\frac{\sqrt{T-t}}{\sqrt{T-S}\vee1}$ and $R:=\frac{\Lambda_1(x) R_{0}}{\lambda(x) \cG_t(x)}$ with
\begin{align}\label{HH1}
\cG_t(x):= (\sqrt{T-S}\wedge 1)\left[\frac{ \Lambda_{1}(x)[a]^{1/\alpha}_{L^\infty\bC^\alpha(Q_{R_{0}}(t,x))}}{\lambda^{1+1/\alpha}(x)}
+\left(\frac{\|b\|_{\mL^{q_b}_{p_b}(Q_{R_0}(t,x))}}{\lambda(x)/\Lambda_1^{1/q_b}(x)}\right)^{1/\theta_b}\right]+\frac{\Lambda_{1}(x)}{\lambda(x)}.
\end{align}
\end{theorem}

\begin{remark}
Suppose that $a(t,x)=\lambda I_d$ and $b\in L^{q_b}_{\text{loc}}\big(\mathbb{R}; L^{p_b}_{\text{loc}}(\mathbb{R}^{d})\big)$ for some $p_b,q_b\in(1,\infty]$ with $d/p_b+2/q_b<1$. Note that
\[
\|u\|_{\mL^q_tL^p_x(Q_{R}(t,x))}\lesssim R^{d/p+2/q}\|u\|_{\mL^\infty(Q_{R}(t,x))}=R^{1-\theta}\|u\|_{\mL^\infty(Q_{R}(t,x))}.
\]
By \eqref{grad0}, there exists $C=C(d,p,q,p_b,q_b,\lambda)>0$ such that for every
$(t,x)\in(\frac {T+S}2,T)\times\mathbb{R}^d$,
\[
|\nabla_{x} u(t,x)|\lesssim_{C}R_0^{-1}\left(\|b\|^{1/\theta_b}_{\mL^{q_b}_{p_b}(Q_{R_0}(t,x))}+1\right)\|u\|_{\mL^\infty(Q_{R_0}(t,x))}
+R_0\|f\|_{\mL^\infty(Q_{R_0}(t,x))},
\]
where $R_0:=\sqrt{T-t}/(\sqrt{T-S}\vee1)$ and $\theta_b=1-d/p_b-2/q_b$.
It is worth noting that by linearization, the above estimate could be used to derive some a priori regularity estimates for HJB equations (see \cite{ZZZ22}).
\end{remark}

When the drift $b$ is locally H\"older continuous, we obtain the following pointwise Hessian and Schauder estimates.

\begin{theorem}[Hessian estimate]\label{grad3}
Suppose that {\bf (H$^\alpha_a$)} holds and $b\in L^\infty_{\mathrm{loc}}\big(\mathbb{R};\mathbf{C}^\alpha_{\mathrm{loc}}(\mathbb{R}^{d})\big)$ with the same $\alpha$ as in {\bf (H$^\alpha_a$)}.
For given $S<T$,
there exists $C=C(d,\alpha)>0$ such that for any solution $u \in L^{\infty}_{\rm loc}\big((S,T);\mathbf{C}_{\mathrm{loc}}^{2+\alpha}(\mathbb{R}^{d})\big)$ of PDE \eqref{pde31} and $(t, x)\in [\frac {T+S}2,T)\times \mathbb{R}^{d}$,
\begin{align}\label{Ds0}
|\nabla_{x}^{2}u(t,x)|\lesssim_C\left(\frac{\Lambda_1(x)}{\lambda(x)R}\right)^{2}
\|u\|_{L^\infty(Q_{R}(t,x))}
+\frac{\|f\|_{L^\infty\bC^\alpha(Q_{R}(t,x))}}{\lambda(x)},
\end{align}
where $\Lambda_1(x):=\Lambda(x)+1$, $R_{0}:=\frac{\sqrt{T-t}}{\sqrt{T-S}\vee1}$ and $R:=\frac{\Lambda_{1}(x)R_0}{\lambda(x)\cH(x)}$ with
\begin{align}\label{HH2}
\cH(x):=(\sqrt {T-S}\wedge 1)\left[
\frac{ \Lambda_{1}(x)[a]^{1/\alpha}_{L^\infty([S,T];\bC^\alpha(B_1(x)))}}{\lambda^{1+1/\alpha}(x)}
+\frac{\|b\|_{L^\infty([S,T];\bC^\alpha(B_1(x)))}}  {\lambda(x)}\right]+\frac{\Lambda_{1}(x)}{\lambda(x)}.
\end{align}
If, in addition, there exist constants $c_0\in(0,1]$ such that for all $x\in\mathbb{R}^d$,
\begin{align}\label{LL2}
c_0\lambda(x)\leq \inf_{y\in B_1(x)}\lambda(y)\leq\sup_{y\in B_1(x)}\Lambda(y)\leq c_0^{-1}\Lambda(x),
\end{align}
then there exists $C=C(d,\alpha,c_0)>0$ such that for every $(t, x)\in [\frac {T+S}2,T)\times \mathbb{R}^{d}$,
\begin{align}\label{Da3}
\begin{split}
[\nabla_{x}^{2}u(t,\cdot)]_{\bC^\alpha(B_{1/2}(x))}
&\lesssim_C \left(\frac{\Lambda_1(x)}{\lambda(x)R}\right)^{2+\alpha}{\|u\|_{L^\infty(I_{R}(t)\times B_1(x))}}\\
&+\frac{\|f\|_{L^\infty(I_{R}(t)\times B_1(x))}+R^\alpha[f]_{L^\infty(I_{R}(t); \mathbf{C}^{\alpha}(B_1(x)))}}{\lambda(x)R^\alpha}.
\end{split}
\end{align}
\end{theorem}

\subsection{Local $L^q_tL^p_x$-maximal regularity and Schauder's estimates}

The purpose of this subsection is to establish a priori local Sobolev (maximal $L^q_tL^p_x$-regularity) and Schauder-type estimates for solutions to \eqref{pde31}. These estimates will later be combined with embedding arguments to derive pointwise bounds on $\nabla u$ and $\nabla^2 u$.

Let $a(t):\mathbb{R}\to\mathbb{R}^d\otimes\mathbb{R}^d$ be a bounded, symmetric, positive-definite matrix-valued function. For $t<s$ define
\[
X_{t,s}:=\int^s_t \sqrt{2a(r)}\,\dif W_r,
\]
where $(W_r)_{r\in\mathbb{R}}$ is a Brownian motion, and for $\varphi\in C^2_b(\mathbb{R}^d)$,
\[
P_{t,s}\varphi(x):=\mathbb{E}\,\varphi(x+X_{t,s}).
\]
It\^o's formula yields, for Lebesgue almost every $s>t$,
\[
\partial_sP_{t,s}\varphi(x)=\tr\big(a(s)\cdot\nabla^2_xP_{t,s}\varphi(x)\big),
\]
and for Lebesgue almost every $t<s$,
\[
\partial_tP_{t,s}\varphi(x)=-\tr\big(a(t)\cdot\nabla^2_xP_{t,s}\varphi(x)\big).
\]
Moreover, if $f\in C_c(\mathbb{R}; C^2_b(\mathbb{R}^{d}))$ and we set for $t\in\mathbb{R}$,
\[
u(t,x):=\int^{\infty}_tP_{t,s}f(s, x)\dif s=\int^{\infty}_t\mathbb{E}\, f(s, x+X_{t,s})\dif s,
\]
then $u$ satisfies the following equation on $\mathbb{R}\times\mathbb{R}^d$:
\begin{align}\label{pde32}
\partial_{t}u(t,x)+\tr\big(a(t)\cdot\nabla^2_xu(t,x)\big)+f(t,x)=0.
\end{align}

We first employ probabilistic methods to study solutions of \eqref{pde32}. The novelty lies in the explicit dependence of the estimates on the ellipticity constant.

\begin{lemma}\label{lemma31}
Assume that for some $0< \lambda\leq\Lambda<\infty$,
\[
\lambda |\xi|^{2} \leq \langle a(t)\xi,\xi\rangle \leq \Lambda |\xi|^{2}, \qquad \forall \, t\in \mathbb{R},\ \xi\in\mathbb{R}^d.
\]
\begin{enumerate}[(i)]
\item \textbf{($L^q_tL^p_x$-maximal regularity estimate)}
For any $q,p\in(1,\infty)$ there exists a constant $C=C(d,q,p)>0$ such that for all $f\in L^q(\mathbb{R}; L^p(\mathbb{R}^d))$,
\begin{align}\label{aaa}
\|\nabla_{x}^2 u\|_{q,p}\lesssim_C \lambda^{-1}\|f\|_{q,p}.
\end{align}

\item \textbf{(Schauder estimate)}
For any $\alpha\in(0,1)$ there exists a constant $C=C(d,\alpha)>0$ such that for all $f\in C_c(\mathbb{R}; \bC^2(\mathbb{R}^{d}))$,
\begin{align}\label{aa0}
[\nabla_{x}^2 u]_{\infty,\alpha}\lesssim_{C} \lambda^{-1}[f]_{\infty,\alpha},
\end{align}
where $[f]_{\infty,\alpha}:=\sup_{t\in \mathbb{R}}[f(t,\cdot)]_{\mathbf{C}^{\alpha}(\mathbb{R}^{d})}$.
\end{enumerate}
\end{lemma}

\begin{proof}
(i) By a standard density argument we may assume $f\in C_c(\mathbb{R}; C^2_c(\mathbb{R}^{d}))$.
First consider the case $a(t)=\lambda I_d$ with $\lambda>0$. Define
\[
u_\lambda(t,x):=u(t, \sqrt{\lambda}x),\qquad f_\lambda(t,x):=f(t,\sqrt{\lambda}x).
\]
The chain rule gives
\[
\partial_t u_\lambda(t,x)+\Delta u_\lambda(t,x)+f_\lambda(t,x)=0.
\]
By \cite[p.107, Theorem~7]{Kr08} there exists a constant $C=C(d,q,p)>0$ such that
\[
\|\nabla_{x}^2 u_\lambda\|_{q,p}\lesssim_ C\|f_\lambda\|_{q,p}.
\]
A change of variables then yields, for every $\lambda>0$,
\begin{align}\label{bb0}
\|\nabla_{x}^2 u\|_{q,p}\lesssim_ C \lambda^{-1}\|f\|_{q,p}.
\end{align}

Now let $W^{(1)}_t$ and $W^{(2)}_t$ be two independent $d$-dimensional standard Brownian motions. For $t<s$ set
\[
X^{(1)}_{t,s}:=\int^s_t \sqrt{2\bigl(a(r)-\lambda I_d\bigr)}\,\dif W^{(1)}_r,\quad
X^{(2)}_{t,s}:=\sqrt{2\lambda}\,(W^{(2)}_s-W^{(2)}_t).
\]
Since the ellipticity lower bound implies $a(r)-\lambda I_d$ is positive semidefinite, $\sqrt{a(r)-\lambda I_d}$ is well-defined.
For $\varphi\in C^2_b(\mathbb{R}^d)$, we also set
\[
P^{(i)}_{t,s}\varphi(x):=\mathbb{E}\,\varphi(x+X^{(i)}_{t,s}),\quad i=1,2.
\]
Since $X^{(1)}_{t,s}$ and $X^{(2)}_{t,s}$ are independent,
\[
P_{t,s}\varphi(x)=P^{(1)}_{t,s}P^{(2)}_{t,s}\varphi(x).
\]
Hence
\[
u(t,x)=\int^{\infty}_t P_{t,s}f(s, x)\dif s
      =\int^{\infty}_t \mathbb{E}\Bigl(P^{(2)}_{t,s}f\bigl(s, x+X^{(1)}_{0,s}-X^{(1)}_{0,t}\bigr)\Bigr)\dif s
      =:\mathbb{E}\,\widetilde u\bigl(t,x-X^{(1)}_{0,t}\bigr),
\]
where we adopt the convention $X_{0,s}:=-X_{s,0}$ for $s<0$, and
\[
\widetilde u(t,x):=\int^{\infty}_t P^{(2)}_{t,s}f\bigl(s, x+X^{(1)}_{0,s}\bigr)\dif s.
\]
Using Fubini's theorem and estimate \eqref{bb0}, we obtain
\begin{align*}
\int_{\mathbb{R}}\|\nabla^2_x u(t)\|^q_p\dif t
&\leq \mathbb{E}\!\left(\int_{\mathbb{R}}\| \nabla^2_x\widetilde u(t,\cdot+X^{(1)}_{0,t})\|^q_p\dif t\right)
   =\mathbb{E}\!\left(\int_{\mathbb{R}}\| \nabla^2_x\widetilde u(t)\|^q_p\dif t\right)\\
&\lesssim\lambda^{-q}\,\mathbb{E}\!\left(\int_{\mathbb{R}}\|f(s, \cdot+X^{(1)}_{0,s})\|_p^q\dif s\right)
   =\lambda^{-q}\|f\|^q_{q,p},
\end{align*}
which proves \eqref{aaa}.

(ii) Estimate \eqref{aa0} follows by the same decomposition argument, combined with the Schauder estimate for the constant diffusion $\lambda I_d$ and the stability under translations; see \cite[Theorem~3.2]{HWZ20}. The factor $\lambda^{-1}$ is obtained via the same scaling as in (i).
\end{proof}

\begin{remark}
The case of variable coefficients has also been treated in \cite{Kr08} and \cite{HWZ20}; however, the explicit dependence on $\lambda$ is not provided there.
\end{remark}

Fix a point $(t,x)\in \mathbb{R}\times \mathbb{R}^{d}$ and recall the parabolic cylinder $Q_R(t,x)$ with centre $(t,x)$ and radius $R>0$ defined in \eqref{QR1}. For brevity, we shall write for $q,p\in[1,\infty]$,
\begin{align*}
\|f\|_{q,p;R}:=\|f\|_{\mL^q_tL^p_x(Q_R(t,x))}.
\end{align*}
We now prove the following a-priori local $L^q_tL^p_x$-maximal regularity estimate.

\begin{theorem}[Local $L^q_tL^p_x$-maximal regularity estimate]\label{sobov}
Suppose that $b\in L^{q_b}_{\rm loc}(\mathbb{R}; L^{p_b}_{\rm loc}(\mathbb{R}^d))$ for some $q_b,p_b\in(1,\infty]$ with
$\theta_b:=1-\frac {d}{p_b}-\frac {2}{q_b}>0$. 
Fix $R>0$ and $(t,x)\in \mathbb{R} \times \mathbb{R}^{d}$.
We also assume that for some $0<\lambda\leq \Lambda<\infty$,
\[
\lambda|\xi|^2\leq \langle a(s,x)\xi, \xi\rangle\leq	\Lambda|\xi|^2,\quad \forall s\in I_R(t),\ \xi\in\mathbb{R}^{d},
\]
and that for a sufficiently small $\delta=\delta(d,p,q)\in(0,1)$,
\begin{align}\label{As1}
\sup_{(s,y)\in Q_{R}(t,x)}\|a(s,y)-a(s,x)\|_{\mathrm{HS}}\leq \delta \lambda.
\end{align}
For any
$q\in(1,\infty)\cap(1,q_b]$ and $p\in(1,\infty)\cap(1,p_b]$,
there is a constant $C=C(d,p,q,p_b,q_b)>0$ such that for each solution $u\in \mW^{1,2}_{q,p}(Q_R(t,x))$ of equation \eqref{pde31} in the sense of Definition~\ref{Def31},
\begin{align}\label{essob}
\|\nabla^2_x u\|_{q,p;{R/2}}\lesssim_{C}
\left[\frac{\Lambda_1}{R\lambda}+
\Bigg(\frac{\|b\|_{q_b,p_b;R}}{\lambda/\Lambda_1^{1/q_b}}\Bigg)^{1/\theta_b}\right]^{2}\|u\|_{q,p;R}+\frac{\|f\|_{q,p;{R}}}{\lambda},	
\end{align}
where $\Lambda_1:=\Lambda+1$. Moreover, if $\theta:=1-\tfrac d{p}-\tfrac 2{q}>0$, then
\begin{align}\label{essob9}
\|\nabla_x u\|_{\infty;R/2}
\lesssim_C
 \Lambda_1^{1/q}\left[\frac{\Lambda_1}{R\lambda}+
\Bigg(\frac{\|b\|_{q_b,p_b;R}}{\lambda/\Lambda_1^{1/q_b}}\Bigg)^{1/\theta_b}\right]^{2-\theta}\|u\|_{q,p;R}
+\frac{R^{\theta}\|f\|_{q,p;R}}{\lambda/\Lambda_1^{1/q}}.
\end{align}
\end{theorem}

\begin{proof}
We employ the freezing coefficients technique and divide the proof into five steps.

\textbf{Step 1. Freezing at $(t, x)=(0,0)$ and scaling to $R=1$.}
Rewrite \eqref{pde31} as
\begin{align}\label{PDE22}
\partial_tu+\tr\big(a(t,0)\cdot\nabla_{x}^2 u\big)+g=0,
\end{align}
with
\begin{align}\label{Ds01}
g=\tr\big(\big(a(t,\cdot)-a(t,0)\big)\cdot\nabla_{x}^2 u\big)+b\cdot\nabla_xu+f.
\end{align}
Fix $\frac{1}{4}\leq\sigma<\tau\leq1$.
Let $\eta_0(t,x)$ and $\eta_1(t,x)$ be smooth cutoff functions supported in $Q_{(\tau+\sigma)/2}$ and $Q_{\tau}$, respectively, satisfying
\[
0\leq \eta_0 \leq 1\ \text{in } Q_{(\tau+\sigma)/2},\quad \eta_0\equiv 1\ \text{in } Q_{\sigma},
\]
\[
0\leq \eta_1 \leq 1\ \text{in } Q_{\tau},\quad \eta_1\equiv 1\ \text{in } Q_{(\tau+\sigma)/2},
\]
and, for some $C=C(d)>0$,
\begin{align}\label{DS01}
|\nabla_{x} \eta_i|\leq \frac{C}{\tau-\sigma},\quad |\nabla_{x}^2 \eta_i|+ |\partial_t\eta_i|\leq \frac{C}{(\tau-\sigma)^2},\quad i=0,1.
\end{align}
Define $\widetilde u:=u\eta_0$. Multiplying \eqref{PDE22} by $\eta_0$ gives
\[
\partial_t \widetilde u+\tr\big(a(t,0)\cdot\nabla_{x}^2 \widetilde u\big)+\widetilde g=0,
\]
where
\begin{align}\label{Be1}
\widetilde g=g\eta_0-\big(\partial_t\eta_0+\tr\big(a(t,0)\cdot\nabla_{x}^2 \eta_0\big)\big) u-2a(t,0) \nabla\eta_0\cdot\nabla_{x} u.
\end{align}
Set $\upsilon:=(\tau+\sigma)/2$ and $\Lambda_1:=\Lambda+1$.
Since $\eta_0\equiv1$ on $Q_\sigma$ and ${\rm supp}(\eta_0)\subset Q_\upsilon$, we have
\begin{align*}
\|\nabla_{x}^2 u\|_{q,p;\sigma }
&=\|\nabla_{x}^2 \widetilde u\|_{q,p;\sigma }\leq
\|\nabla_{x}^2 \widetilde u\|_{q,p}
\stackrel{\eqref{aaa}}{\lesssim} \lambda^{-1}\|\widetilde g\|_{q,p }
=\lambda^{-1}\|\widetilde g\|_{q,p;\upsilon}\\
&\!\!\!\stackrel{\eqref{DS01}}{\lesssim}\lambda^{-1}\Big(
\|g\|_{q,p;\upsilon }+\Lambda_1(\tau-\sigma)^{-2}\|u\|_{q,p;\upsilon }+\Lambda_1(\tau-\sigma)^{-1} \|\nabla_{x} u\|_{q,p;\upsilon }\Big).
\end{align*}
By the Gagliardo--Nirenberg inequality and Young's inequality, we have for any $\eps_0>0$,
\begin{align}\label{GN1}
\|\nabla_{x} u\|_{q,p;\upsilon }
\leq\eps_0\|\nabla^2_{x} u\|_{q,p;\upsilon}+C\eps_0^{-1}\|u\|_{q,p;\upsilon}.
\end{align}
 Moreover, letting $\frac{1}{q_1}+\frac1{q_b}=\frac1q$ and $\frac{1}{p_1}+\frac1{p_b}=\frac1p$,
by \eqref{Ds01}, \eqref{As1} and H\"older's inequality, we have
\begin{align*}
\|g\|_{q,p;\upsilon}
&\leq\delta\lambda\|\nabla_{x}^2u\|_{q,p;\upsilon}+\|b\cdot\nabla_{x} u\|_{q,p;\upsilon}+\|f\|_{q,p;\upsilon }\no\\
&\leq\delta\lambda\|\nabla_{x}^2u\|_{q,p;\upsilon}+\|b\|_{q_b,p_b;\upsilon}\|\nabla_{x} u\|_{q_1,p_1;\upsilon}+\|f\|_{q,p;\upsilon }.
\end{align*}
Combining the above estimates and using $\lambda/\Lambda_1\leq 1$, we find a constant $C=C(d,p,q)>0$ such that
\begin{align}
\begin{split}\label{DS22}
\|\nabla_{x}^2 u\|_{q,p;\sigma }
&\lesssim_C
(\delta+\eps_0)\|\nabla^2_{x} u\|_{q,p;\upsilon }+\eps_0^{-1}(\Lambda_1/\lambda)^2(\tau-\sigma)^{-2}\|u\|_{q,p;\upsilon }\\
&\quad+\lambda^{-1}\|b\|_{q_b,p_b;1}{ \|\nabla_{x} u\|_{q_1,p_1;\upsilon}} +\lambda^{-1}\|f\|_{q,p;\upsilon }.
\end{split}
\end{align}

\textbf{Step 2. Estimates of $\|\nabla_{x} u\|_{q_1,p_1;\upsilon}$.}
Since $\eta_1\equiv1$ on $Q_\upsilon$ and ${\rm supp}(\eta_1)\subset Q_\tau$,
\[
\|\nabla_x u\|_{q_1,p_1;\upsilon}\leq \|\nabla_x (u\eta_1)\|_{q_1,p_1}.
\]
Recall that $\theta_b:=1-\frac {d}{p_b}-\frac {2}{q_b}$.
By Lemma~\ref{inter1} with $(q_0,p_0)=(q_1,p_1)$,
there exists a constant $C=C(d,p,q,p_b,q_b)>0$ such that for any $\eps_1, \eps_2>0$,
\begin{align*}
\|\nabla_x (u\eta_1)\|_{q_1,p_1}
&\lesssim_{C} \eps_{1}\|\nabla^2_x (u\eta_1)\|_{q,p}+\eps_2 \|\partial_t(u\eta_1)\|_{q,p}
+ \eps_1^{(-1-d/p_b)/\theta_{b}}\eps_2^{-2/(q_b\theta_b)}\|{u\eta_1}\|_{q,p}.
\end{align*}
Note that if $q_b=\infty$, one can take $\eps_2\equiv 0$.
Using \eqref{DS01}, \eqref{GN1} and ${\rm supp}(\eta_1)\subset Q_\tau$, we have
\begin{align*}
\|\nabla^2_x (u\eta_1)\|_{q,p}
&\leq\|\nabla^2_xu\|_{q,p;\tau}+2\|\nabla_x\eta_1\|_\infty\|\nabla_x u\|_{q,p;\tau}+\|u\|_{q,p;\tau}\\
&\leq 2\|\nabla^2_xu\|_{q,p;\tau}+C(\tau-\sigma)^{-2}\|u\|_{q,p;\tau},
\end{align*}
and by equation \eqref{pde31},
\begin{align*}
\|\partial_t (u \eta_1)\|_{q,p}&\leq\Lambda_1\|\nabla^2_x u\|_{q,p;\tau}+\|b\|_{q_b,p_b;\tau}\|\nabla_{x} u\|_{q_1,p_1;\tau}\no\\
&\quad+C(\tau-\sigma)^{-2}\|u\|_{q,p;\tau}+\|f\|_{q,p;\tau}.
\end{align*}
Hence, for some $C_1=C_1(d,p,q,p_b,q_b)>0$ and all $\eps_1,\eps_2>0$,
\begin{align*}
\|\nabla_x u\|_{q_1,p_1;\upsilon}
&\lesssim_{C_1} \big[\eps_{1}+\eps_2\Lambda_1\big]\|\nabla^2_x {u}\|_{q,p;{\tau}}
+\eps_{2}(\|b\|_{q_b,p_b;1}\|\nabla_{x} u\|_{q_1,p_1;\tau}+\|f\|_{q,p;\tau})\no\\
&\quad+ \big(\eps_1^{(-1-d/p_b)/\theta_{b}}\eps_2^{-2/(q_b\theta_b)}+(\eps_1+\eps_2)(\tau-\sigma)^{-2}\big)\|{u}\|_{q,p;{\tau}}.
\end{align*}
In particular, if we choose $\eps_2=\eps_1/\Lambda_1\leq(\eps_1/\lambda)\wedge\eps_1$, then by $1+\frac{d}{p_b}+\frac{2}{q_b}=2-\theta_b $,
\begin{align}\label{case11}
\|\nabla_x u\|_{q_1,p_1;\upsilon}
&\lesssim_{C_1} \eps_{1}\Big[\|\nabla^2_x {u}\|_{q,p;{\tau}}
+(\|b\|_{q_b,p_b;1}\|\nabla_{x} u\|_{q_1,p_1;\tau}+\|f\|_{q,p;\tau})/\lambda\no\\
&\qquad\qquad+ \big(\eps_1^{-2/\theta_{b}}\Lambda_1^{2/(q_b\theta_b)}+(\tau-\sigma)^{-2}\big)\|{u}\|_{q,p;{\tau}}\Big].
\end{align}

\textbf{Step 3. Estimate of $\|\nabla_{x}^2 u\|_{q,p;3/4}$.}
Assume first $\|b\|_{q_b,p_b;1}\neq0$. For any $\eps>0$, if we choose $\eps_1=\frac{\eps\lambda}{\|b\|_{q_b,p_b;1}}$ in \eqref{case11},
then
\begin{align}\label{k3}
\begin{split}
\lambda^{-1}\|b\|_{q_b,p_b;1}\|\nabla_x u\|_{q_1,p_1;\upsilon}
&\lesssim_{C_1}
\eps\Bigg[\|\nabla^2_x {u}\|_{q,p;{\tau}}
+\frac{\|b\|_{q_b,p_b;1} \|\nabla_x u\|_{q_1,p_1;\tau}+\|f\|_{q,p;\tau}}{\lambda}\Bigg]\\
&\qquad+\eps\Bigg[\Bigg(\frac{\|b\|_{q_{b},p_b;1}}
{\eps\lambda/\Lambda_1^{1/q_b}}\Bigg)^{2/\theta_b}+\frac{1}{(\tau-\sigma)^{2}}\Bigg]\|{u}\|_{q,p;{\tau}}.
\end{split}
\end{align}
Inserting this into \eqref{DS22} yields for some $C_2=C_2(d,p,q)>0$ and any $\eps_0,\eps>0$,
\begin{align}
\label{k31}
\begin{split}
\|\nabla_{x}^2 u\|_{q,p;\sigma }
&\lesssim_{C_2}
\big(\delta+\eps_0+\eps\big)\|\nabla_{x}^2u\|_{q,p;\tau }
+\frac{\eps\|b\|_{q_{b},p_b;1} \|\nabla_x u\|_{q_1,p_1;\tau}}{\lambda}
+\frac{\|f\|_{q,p;\tau}}\lambda\\
&\qquad+\Bigg[\eps\Bigg(\frac{\|b\|_{q_{b},p_b;1}}{\eps\lambda/\Lambda_1^{1/q_{b}}}\Bigg)^{2/\theta_b}+
\frac{\eps+(\Lambda_1/\lambda)^2/\eps_0}{(\tau-\sigma)^2}\Bigg]\|{u}\|_{q,p;{\tau}}.
\end{split}
\end{align}
Define
\[
h(\sigma):=\|\nabla_{x}^2 u\|_{q,p;\sigma }+\lambda^{-1}\|b\|_{q_{b},p_b;1}\|\nabla_x u\|_{q_1,p_1;\sigma}.
\]
Then \eqref{k3} and \eqref{k31} imply the existence of $C_3=C_3(d,p,q)>0$ such that for all $\frac12\leq\sigma<\tau\leq 1$ and $\delta,\eps_0,\eps\in(0,1)$,
\begin{align*}
h(\sigma)
&\lesssim_{C_3}
\big(\delta+\eps_0+\eps\big)h(\tau)+\frac{\|f\|_{q,p;\tau}}\lambda
+\Bigg[\Bigg(\frac{\|b\|_{q_{b},p_b;1}}{\eps\lambda/\Lambda_1^{1/q_{b}}}\Bigg)^{2/\theta_b}+
\frac{(\Lambda_1/\lambda)^2}{\eps_0(\tau-\sigma)^2}\Bigg]\|{u}\|_{q,p;{\tau}}.
\end{align*}
Now, choosing $\delta=\eps_0=\eps=\frac1{8C_3}$, we obtain for some $C_4=C_4(d,p,q)>0$ and all $\tfrac12\leq\sigma<\tau\leq 1$,
$$
h(\sigma)\leq \frac12h(\tau)+\frac{C_3\|f\|_{q,p;1}}{\lambda}+\frac{C_4\sA\|u\|_{q,p;1}}{(\tau-\sigma)^2},
$$
where
\[
\sA:=\Bigg(\frac{\|b\|_{q_{b},p_b;1}}{\lambda/\Lambda_1^{1/q_b}}\Bigg)^{2/\theta_b}+\frac{\Lambda_1^2}{\lambda^2}.
\]
Lemma~\ref{Le26} then gives
\begin{align}\label{As22}
\|\nabla_{x}^2 u\|_{q,p;3/4}+  \frac{ \|b\|_{q_{b},p_b;1}\|\nabla_x u\|_{q_1,p_1;3/4}}\lambda
= h\left(\tfrac34\right)\lesssim\sA\|u\|_{q,p;1}+\frac{\|f\|_{q,p;1}}{\lambda}.
\end{align}
If $\|b\|_{q_b,p_b;1}=0$, a similar argument yields
\begin{align*}
\|\nabla_{x}^2 u\|_{q,p;3/4}
\lesssim_{C_1}
\frac{\Lambda^2_1}{\lambda^2}\|{u}\|_{q,p;1}+
\frac{\|f\|_{q,p;1}}{\lambda}.
\end{align*}

\textbf{Step 4. Estimate of $\|\nabla_x u\|_{\infty;1/2}$.}
Suppose that $\theta=1-\tfrac d{p}-\tfrac 2{q}>0.$
As in step 2 of proving \eqref{case11},
if we use Lemma~\ref{inter1} with $(q_0,p_0)=(\infty,\infty)$, then
there is a constant $C_2=C_2(d,p,q,q_b,p_b)>0$ such that for all $\eps_1>0$,
\begin{align}\label{case1}
\begin{split}
\|\nabla_x u\|_{\infty;1/2}
&\lesssim_{C_2} \eps_1\Big[\|\nabla^2_x {u}\|_{q,p;{3/4}}
+\lambda^{-1}\|b\|_{q_b,p_b;1}\|\nabla_{x} u\|_{q_1,p_1;3/4}\\
&\qquad+\|f\|_{q,p;3/4}/\lambda+ \big(\eps_1^{-2/\theta}\Lambda_1^{2/(q\theta)}+1\big)\|{u}\|_{q,p;3/4}\Big],
\end{split}
\end{align}
where $\theta:=1-\frac{d}{p}-\frac{2}{q}$.
Inserting \eqref{As22} into \eqref{case1}, we get
\begin{align*}
\|\nabla_x u\|_{\infty;1/2}
&\lesssim \eps_1\left(\sA\|u\|_{q,p;1}+\lambda^{-1} {\|f\|_{q,p;1}}+\eps_1^{-2/\theta}\Lambda_1^{2/(q\theta)}\|{u}\|_{q,p;1}\right).
\end{align*}
Finally, taking $\eps_1=\Lambda_{1}^{1/q}\sA^{-\theta/2}$ and using $\sA\geq\Lambda_1^2/\lambda^2\geq 1$ yield
\begin{align}\label{As221}
\|\nabla_x u\|_{\infty;1/2}
&\lesssim
\Lambda_1^{1/q}\sA^{1-\theta/2}\|{u}\|_{q,p;1}+
\Lambda_1^{1/q}\|f\|_{q,p;1}/(\lambda\sA^{\theta/2})\no\\
&\leq\Lambda_1^{1/q}\left[\frac{\Lambda_1}{\lambda}+
\Bigg(\frac{\|b\|_{q_b,p_b;1}}{\lambda/\Lambda_1^{1/q_b}}\Bigg)^{1/\theta_b}\right]^{2-\theta}\|u\|_{q,p;1}
+\frac{\|f\|_{q,p;1}}{\lambda/\Lambda_1^{1/q}}.
\end{align}

\textbf{Step 5. Parabolic scaling.} For general $t,x$ and $R>0$, define
\begin{align}\label{tr1}
b_R(s,y)&:=R\, b(R^2s+t, R y+x),\quad
a_R(s,y):=a(R^2s+t, R y+x),\\[2mm]
\label{tr2}
u_R(s,y)&:=u(R^2s+t, R y+x),\quad
f_R(s,y):=R^2 f(R^2s+t, R y+x).
\end{align}
If $u$ satisfies \eqref{pde31} in $Q_{R}(t,x)$, the chain rule shows that $u_R$ satisfies \eqref{pde31} in $Q_1:=Q_{1}(0,0)$ with coefficients $b_{R}, a_{R}$ and $f_{R}$, and condition \eqref{As1} becomes
\[
\sup_{(s,y)\in Q_{1}}|a_{R}(s,y)-a_{R}(s,0)|\leq \delta \lambda.
\]
Applying the change of variables and the a priori estimate \eqref{As22}, we get
\begin{align*}
\|\nabla^2_x u\|_{q,p;R/2}
&=R^{d/p+2/q-2}\|\nabla^2_x u_R\|_{q,p;1/2}
=\frac{\|\nabla^2_x u_R\|_{q,p;1/2}}{R^{1+\theta}}\\
&\lesssim \frac{1}{R^{1+\theta}}\Bigg(\Bigg[\frac{\Lambda_1^2}{\lambda^2}+
\Bigg(\frac{\|b_R\|_{q_b,p_b;1}}{\lambda/\Lambda_1^{1/q_b}}\Bigg)^{2/\theta_b}\Bigg]\|u_R\|_{q,p;1}
+\frac{\|f_R\|_{q,p;1}}\lambda\Bigg)\\
&= \Bigg[\frac{\Lambda_1^2}{R^2\lambda^2}+\Bigg(\frac{\|b\|_{q_b,p_b;R}}{\lambda/\Lambda_1^{1/q_b}}\Bigg)^{2/\theta_b}\Bigg]\|{u}\|_{q,p;R} +
\frac{\|{f}\|_{q,p;R}}{\lambda},
\end{align*}
where we have used
\[
\|u_R\|_{q,p;1}=R^{\theta-1}\|u\|_{q,p;R},\quad
\|b_R\|_{q_b,p_b;1}=R^{\theta_b}\|b\|_{q_b,p_b;R},\quad
\|{f}_R\|_{q,p;1}=R^{1+\theta}\|{f}\|_{q,p;R}.
\]
Moreover, if $\theta=1-\tfrac d{p}-\tfrac 2{q}>0,$ then from \eqref{As221} we deduce
\begin{align*}
\|\nabla_x u\|_{\infty;R/2}
&\lesssim
\frac{\Lambda_1^{1/q}}R\Bigg[\frac{\Lambda_1}{\lambda}+\Bigg(\frac{\|b_R\|_{q_b,p_b;1}}{\lambda/\Lambda_1^{1/q_b}}\Bigg)^{1/\theta_b}\Bigg]^{2-\theta}\|u_R\|_{q,p;1}
+\frac{\|f_R\|_{q,p;1}}{R\lambda/\Lambda_1^{1/q}}\\
&=\Lambda_1^{1/q}\Bigg[\frac{\Lambda_1}{R\lambda}+\Bigg(\frac{\|b\|_{q_b,p_b;R}}{\lambda/\Lambda_1^{1/q_b}}\Bigg)^{1/\theta_b}\Bigg]^{2-\theta}\|u\|_{q,p;R}
+\frac{R^\theta\|f\|_{q,p;R}}{\lambda/\Lambda_1^{1/q}}.
\end{align*}
The proof is complete.
\end{proof}

Fix a point $(t,x)\in \mathbb{R}\times \mathbb{R}^{d}$ and recall the parabolic cylinder $Q_R(t,x)$ with centre $(t,x)$ and radius $R>0$ defined in \eqref{QR1}.
For $\alpha \in (0,1)$ we set
\begin{align}\label{pR2}
[f]_{\infty,\alpha;R}:=\sup_{s\in I_{R}(t)}[f(s,\cdot)]_{\mathbf{C}^{\alpha}(B_R(x))}
\end{align}
and
\begin{align}\label{pR02}
 \|f\|_{\infty;R}:=\sup_{s\in I_{R}(t)}\|f(s,\cdot)\|_{L^\infty(B_R(x))}.
\end{align}
Next we establish an {a priori} local Schauder estimate for strong solutions to \eqref{pde31}.

\begin{theorem}[Local Schauder estimate]\label{Schau}
Let $\alpha\in(0,1)$, $R>0$ and $(t,x)\in \mathbb{R} \times \mathbb{R}^{d}$. Assume that
$b\in L^\infty\bC^\alpha(Q_R(t,x))$ and
there exist constants $0<\lambda\leq \Lambda<\infty$ such that
\[
\lambda|\xi|^2\leq \langle a(s,x)\xi, \xi\rangle\leq\Lambda|\xi|^2,\ s\in I_R(t),\ \xi\in\mathbb{R}^{d},
\]
and that for a sufficiently small $\delta=\delta(d,\alpha)\in(0,1)$,
\begin{align}\label{ABs1}
R^{\alpha}[a]_{\infty,\alpha;R}\leq \delta \lambda.
\end{align}
Then there exists a constant $C=C(d,\alpha)>0$ such that for every
$u\in L^\infty_{\rm loc}\big(I_R(t);\mathbf{C}_{\rm loc}^{2+\alpha}(B_R(x))\big)$ solving \eqref{pde31} in $Q_R(t,x)$,
\begin{align*}
[\nabla_{x}^{2}u]_{\infty,\alpha;{R/2}}
\lesssim_C
\Bigg(\frac{\Lambda_1/R+\|b\|_{\infty;R}+R^{\alpha}[b]_{\infty,\alpha;R}}{\lambda}\Bigg)^{2+\alpha}\|u\|_{\infty;R}
+\frac{\|f\|_{\infty;R}+R^{\alpha}[f]_{\infty,\alpha;R}}{R^\alpha\lambda},
\end{align*}
where $\Lambda_1:=\Lambda+1$ and $[\cdot]_{\infty,\alpha;R},\|\cdot\|_{\infty;R}$ are defined in \eqref{pR2} and \eqref{pR02}.
\end{theorem}

\begin{proof}
\textbf{Step 1. Reduction to $(t,x)=(0,0)$ and $R=1$.}
By the parabolic scaling \eqref{tr1}--\eqref{tr2}, it suffices to prove the estimate for $(t,x)=(0,0)$ and $R=1$.
We fix $\frac12\leq\sigma<\tau\leq1$ and set $\upsilon:=\frac{\sigma+\tau}{2}$.
Let $\eta=\eta_0$ be the standard cutoff functions supported in $Q_\upsilon$, as in Theorem~\ref{sobov}.
Besides \eqref{DS01}, we have for some $C=C(d,\alpha)$,
\begin{align}\label{DS11-bis}
[\eta]_\alpha\leq C(\tau-\sigma)^{-\alpha},\
[\nabla\eta]_\alpha\leq C(\tau-\sigma)^{-1-\alpha},\
[\nabla^2\eta]_\alpha+[\partial_t\eta_0]_\alpha\leq C(\tau-\sigma)^{-2-\alpha}.
\end{align}
Define $\widetilde u:=u\eta$. Then $\widetilde u$ solves on $\mathbb{R}\times\mathbb{R}^d$
\[
\partial_t\widetilde u+\tr(a(t,0)\nabla_x^2\widetilde u)+\widetilde g=0,
\]
where $\widetilde g$ is given by \eqref{Be1} and supported in $Q_\upsilon$.
Applying Lemma~\ref{lemma31}(ii) (the global Schauder estimate \eqref{aa0}) yields
\[
[\nabla_x^2\widetilde u]_{\infty,\alpha}\lesssim \lambda^{-1}[\widetilde g]_{\infty,\alpha}.
\]
Since $\eta\equiv1$ on $Q_\sigma$, we have
\[
[\nabla_x^2 u]_{{\infty,\alpha;\sigma}}
=[\nabla_x^2\widetilde u]_{{\infty,\alpha;\sigma}}
\leq [\nabla_x^2\widetilde u]_{\infty,\alpha}
\lesssim \lambda^{-1}[\widetilde g]_{\infty,\alpha}
\leq \lambda^{-1}[\widetilde g]_{{\infty,\alpha;\upsilon}}.
\]
Thus it remains to estimate $[\widetilde g]_{\infty,\alpha;\upsilon}$.

\textbf{Step 2. Estimating the forcing term $\widetilde g$.}
Using the product estimate for time-uniform H\"older seminorms,
\[
[fg]_{\infty,\alpha;\upsilon}\leq [f]_{\infty,\alpha;\upsilon}\|g\|_{\infty;\upsilon}
+[g]_{\infty,\alpha;\upsilon}\|f\|_{\infty;\upsilon},
\]
together with \eqref{Be1} and \eqref{DS11-bis}, we obtain
\begin{align}\label{gtilde-est-1}
[\widetilde g]_{\infty,\alpha;\upsilon}
\lesssim\; &[g]_{\infty,\alpha;\upsilon}+(\tau-\sigma)^{-\alpha}\|g\|_{\infty;\upsilon}
+\Lambda_1(\tau-\sigma)^{-2}[u]_{\infty,\alpha;\upsilon}
+\Lambda_1(\tau-\sigma)^{-2-\alpha}\|u\|_{\infty;\upsilon}\no\\
&+\Lambda_1(\tau-\sigma)^{-1}[\nabla_x u]_{\infty,\alpha;\upsilon}
+\Lambda_1(\tau-\sigma)^{-1-\alpha}\|\nabla_x u\|_{\infty;\upsilon}.
\end{align}
Next, recall $g=\tr((a-a(\cdot,0))\nabla_x^2u)+b\cdot\nabla_xu+f$.
By \eqref{ABs1} with $R=1$ and $|a(t,x)-a(t,0)|\leq [a]_{\infty,\alpha;1}$ for $|x|\leq1$,
\begin{align*}
[g]_{\infty,\alpha;\upsilon}
\leq\; &\delta\lambda\big([\nabla_x^2u]_{\infty,\alpha;\upsilon}+\|\nabla_x^2u\|_{\infty;\upsilon}\big)
+\|b\|_{\infty;1}[\nabla_xu]_{\infty,\alpha;\upsilon}\\
&+[b]_{\infty,\alpha;1}\|\nabla_xu\|_{\infty;\upsilon}
+[f]_{\infty,\alpha;\upsilon},
\end{align*}
and
$$
\|g\|_{\infty;\upsilon}
\leq\; \delta\lambda\|\nabla_x^2u\|_{\infty;\upsilon}
+\|b\|_{\infty;1}\|\nabla_xu\|_{\infty;\upsilon}
+\|f\|_{\infty;\upsilon}.
$$
Substituting the above two estimates into \eqref{gtilde-est-1}, we get for all $\frac12\leq\sigma<\tau\leq1$,
\begin{align}\label{g-est-3}
\begin{split}
[\widetilde g]_{\infty,\alpha;\upsilon}
&\lesssim\delta\lambda [\nabla_x^2u]_{\infty,\alpha;\upsilon}
+\delta\lambda(\tau-\sigma)^{-\alpha}\|\nabla_x^2u\|_{\infty;\upsilon}\\
&+(\|b\|_{\infty;1}+\Lambda_1(\tau-\sigma)^{-1})[\nabla_xu]_{\infty,\alpha;\upsilon}\\
&+(\tau-\sigma)^{-1-\alpha}([b]_{\infty,\alpha;1}+\|b\|_{\infty;1}+\Lambda_1)\|\nabla_xu\|_{\infty;\upsilon}\\
&+\Lambda_1(\tau-\sigma)^{-2}[u]_{\infty,\alpha;\upsilon}
+\Lambda_1(\tau-\sigma)^{-2-\alpha}\|u\|_{\infty;\upsilon}\\
&+[f]_{\infty,\alpha;\upsilon}+(\tau-\sigma)^{-\alpha}\|f\|_{\infty;\upsilon}.
\end{split}
\end{align}

\textbf{Step 3. Interpolating lower-order terms by $[\nabla_x^2u]_{\infty,\alpha;\tau}$.}
We now invoke Lemma~\ref{inter11}. In particular, there exists a constant $\gamma_*=\gamma_*(d,\alpha)>0$
such that for any $j\in\{0,1,2\}$ and $\beta\in[0,\alpha]$,
\begin{equation}\label{interp-unif}
[\nabla^j_x u]_{\infty,\beta;\upsilon}
\lesssim
\eps^{2-j+\alpha-\beta}[\nabla_x^2u]_{\infty,\alpha;\tau}
+\eps^{-j-\beta}\Big(\eps^{2+\alpha}(\tau-\sigma)^{-{\gamma_*}}+1\Big)\|u\|_{\infty;\tau},
\end{equation}
for all $\eps\in(0,1)$.
Applying \eqref{interp-unif} to $(j,\beta)=(2,0),(1,\alpha),(1,0),(0,\alpha)$ with suitable choices of $\eps$,
and substituting into \eqref{g-est-3}, we arrive at an inequality of the form
\begin{align}\label{main-iter}
[\widetilde g]_{\infty,\alpha;\upsilon}
\leq C_0\Big\{
&(\delta+\eps)\lambda[\nabla_x^2u]_{\infty,\alpha;\tau}
+\Xi\,(\tau-\sigma)^{-\gamma_*}\|u\|_{\infty;\tau}\no\\
&+[f]_{\infty,\alpha;\tau}+(\tau-\sigma)^{-\alpha}\|f\|_{\infty;\tau}\Big\},
\end{align}
where $C_0=C_0(d,\alpha)$ and
\[
\Xi:=\big(\Lambda_1+\|b\|_{\infty;1}+[b]_{\infty,\alpha;1}\big)^{2+\alpha}/\lambda^{1+\alpha}.
\]

\textbf{Step 4. Iteration in $\sigma$.}
Combining \eqref{main-iter} with $[\nabla_x^2u]_{\infty,\alpha;\sigma}\lesssim \lambda^{-1}[\widetilde g]_{\infty,\alpha;\upsilon}$ yields
\[
[\nabla_x^2u]_{\infty,\alpha;\sigma}
\lesssim_{C_1}(\delta+\eps)[\nabla_x^2u]_{\infty,\alpha;\tau}
+\frac{\Xi\,\|u\|_{\infty;\tau}}{\lambda(\tau-\sigma)^{\gamma_*}}
+\frac{(\tau-\sigma)^{-\alpha}\|f\|_{\infty;\tau}+[f]_{\infty,\alpha;\tau}}{\lambda},
\]
for some $C_1=C_1(d,\alpha)$.
Choose $\delta$ and $\eps$ sufficiently small so that $C_1(\delta+\eps)\leq\frac12$.
Then Lemma~\ref{Le26} implies
\begin{align*}
[\nabla_{x}^2 u]_{\infty,\alpha;1/2}
\lesssim
\Bigl(\frac{\Lambda_1+\|b\|_{\infty;1}+[b]_{\infty,\alpha;1}}{\lambda}\Bigr)^{2+\alpha}\|u\|_{\infty;1}
+\frac{\|f\|_{\infty;1}+[f]_{\infty,\alpha;1}}{\lambda}.
\end{align*}

\textbf{Step 5. Scaling back to general $R>0$.}
Applying the parabolic scaling \eqref{tr1}--\eqref{tr2} and using the identities
\[
\|b_R\|_{\infty;1}=R\|b\|_{\infty;R},\qquad [b_R]_{\infty,\alpha;1}=R^{1+\alpha}[b]_{\infty,\alpha;R},
\]
\[
\|f_R\|_{\infty;1}=R^2\|f\|_{\infty;R},\qquad [f_R]_{\infty,\alpha;1}=R^{2+\alpha}[f]_{\infty,\alpha;R},
\]
we obtain exactly the desired estimate on $Q_{R/2}(t_0,x_0)$.
\end{proof}

\subsection{Proofs of Theorems \ref{grad1} and \ref{grad3}}

Throughout this subsection, fix $S<T$ and let $\alpha\in(0,1]$ be as in {\bf (H$^\alpha_a$)}.
For fixed $(t,x)\in[\frac{T+S}{2},T)\times{\mathbb R}^d$, define
\[
R_0:=\frac{\sqrt{T-t}}{\sqrt{T-S}\vee 1},\quad  \Lambda_1:=\Lambda(x)+1,\quad  \lambda:=\lambda(x),
\]
and for $R\in(0,R_0]$,
\[
I_R(t):=(t-R^2,t+R^2),
\]
Note that $0<R_0\leq\sqrt{T-t}$ and $I_{R_0}(t)\subset(S,T)$.

\vspace{1mm}

We now use the estimates in Theorem \ref{sobov} to establish the pointwise gradient estimate.

\begin{proof}[Proof of Theorem \ref{grad1}]
Let $\delta\in(0,1)$ be the small constant in \eqref{As1} (from Theorem~\ref{sobov}).
Recall $\theta_b:=1-d/p_b-2/q_b>0$. Define
\[
\vartheta:=\Big(\frac{[a]_{\infty,\alpha;R_0}}{\delta\lambda}\Big)^{1/\alpha}+
\Big(\frac{\|b\|_{q_b,p_b;R_0}}{\lambda/\Lambda_1^{1/q_b}}\Big)^{1/\theta_b}\frac{\lambda}{\Lambda_{1}}+
\frac{1}{\sqrt{T-S}\wedge 1},
\]
and
$$
R_1:=\frac{R_0}{\vartheta(\sqrt{T-S}\wedge 1)}.
$$
From the very definition, one sees that $R_1\leq\vartheta^{-1}\wedge R_0$. Hence
\[
\sup_{(s,y)\in Q_{R_1}(t,x)}|a(s,y)-a(s,x)|
\leq R_1^{\alpha}[a]_{\infty,\alpha;R_1}
\leq \vartheta^{-\alpha}[a]_{\infty,\alpha;R_0}\leq\delta\lambda,
\]
and
\[
R_1\lambda\Big(\lambda^{-1}\Lambda_1^{1/q_b}\|{b}\|_{q_b,p_b;R_1}\Big)^{1/\theta_b}
\leq \vartheta^{-1}\Lambda_1\Big(\lambda^{-1}\Lambda_1^{1/q_b}\|{b}\|_{q_b,p_b;R_0}\Big)^{1/\theta_b}\leq\Lambda_1.
\]
Thus one may apply \eqref{essob9} with $R=R_1$ to obtain
\begin{align}\label{Re2}
|\nabla_{x} u(t,x)|\leq\|\nabla_x u\|_{\infty;R_1/2}
&\lesssim \Lambda_1^{1/q}\Big(\frac{\Lambda_1}{R_1\lambda}\Big)^{2-\theta}\|{u}\|_{q,p;R_1}
+\frac{R_1^\theta\|{f}\|_{q,p;R_1}}{\lambda/\Lambda_1^{1/q}}.
\end{align}
From the definitions of $\vartheta$ and $\cG_{t}(x)$ (see \eqref{HH1}), it is easy to see that
$$
R_1=\frac{R_0}{\vartheta(\sqrt{T-S}\wedge 1)}\leq \frac{\Lambda_1(x)\,R_0}{\lambda(x)\,\cG_{t}(x)}=:R.
$$
Inserting this into \eqref{Re2} yields precisely the estimate stated in Theorem~\ref{grad1}.
\end{proof}

Using the local Schauder estimate in Theorem~\ref{Schau} we now prove the Hessian estimate.

\begin{proof}[Proof of Theorem~\ref{grad3}]
We divide the proof into two steps.

\vspace{1mm}
\noindent\textbf{Step 1. Pointwise Hessian bound \eqref{Ds0}.}
Let $\delta\in(0,1)$ be the small constant in \eqref{ABs1} (from Theorem~\ref{Schau}).
Define
\[
\vartheta:=\vartheta(x):=\Big(\frac{[a]_{\infty,\alpha;R_0}}{\delta\lambda}\Big)^{1/\alpha}+
\frac{\|b\|_{\infty,\alpha;R_0}}{\Lambda_1}+\frac{1}{\sqrt{T-S}\wedge 1},
\]
and
\[
R_1:=\frac{R_0}{\vartheta(\sqrt{T-S}\wedge 1)}.
\]
Then $R_1\leq\vartheta^{-1}\wedge R_0$ and, by construction,
\[
R_1^{\alpha}[{a}]_{\infty,\alpha;R_1}\leq \vartheta^{-\alpha}[{a}]_{\infty,\alpha;R_0}\leq\delta\lambda,
\qquad
R_1\|{b}\|_{\infty,\alpha;R_1}\leq \vartheta^{-1}\|{b}\|_{\infty,\alpha;R_0}\leq \Lambda_1.
\]
Thus one may apply Theorem \ref{Schau} with $R=R_1$ to obtain
\begin{align}\label{eq:Schauder-on-QR}
[\nabla_{x}^{2}u]_{\infty,\alpha;R_1/2}
&\lesssim\Big(\frac{\Lambda_1}{R_1\lambda}\Big)^{2+\alpha}\|u\|_{\infty;R_1}
+\frac{\|{f}\|_{\infty;R_1}+R_1^{\alpha}[f]_{\infty,\alpha;R_1}}{R_1^\alpha\lambda}.
\end{align}
We now convert \eqref{eq:Schauder-on-QR} into a pointwise bound.
Define the spatial rescaling
\[
u_{R_1}(s,y):=u(s,x+R_1y), \qquad y\in B_{1/2}(0).
\]
Applying the interpolation inequality \eqref{Ho1} with $\sigma=1/4$, $\tau=1/2$, we obtain for any $0<\eps\leq1$,
\begin{align*}
|\nabla^2_xu(t,x)|&\leq\|\nabla^2_xu\|_{\infty;R_1/4}
=R_1^{-2}\|\nabla^2_xu_{R_1}\|_{\infty;1/4}\\
&\lesssim R_1^{-2}\Big(\eps^{-2}\|u_{R_1}\|_{\infty;1/2}+\eps^\alpha[\nabla^2_xu_{R_1}]_{\infty,\alpha;1/2}\Big)\\
&=(R_1\eps)^{-2}\|{u}\|_{\infty;R_1/2}+(R_1\eps)^{\alpha}[\nabla_{x}^{2}{u}]_{\infty,\alpha;R_1/2}\\
&\!\!\stackrel{\eqref{eq:Schauder-on-QR}}{\lesssim} \frac{\|{u}\|_{\infty;{R_1}}}{(R_1\eps)^2}+(R_1\eps)^{\alpha}\Big(\frac{\Lambda_1}{R_1\lambda}\Big)^{2+\alpha}\|{u}\|_{\infty;R_1}\\
&\qquad+\frac{\eps^{\alpha}\big(\|{f}\|_{\infty;{R_1}}+R_1^\alpha[f]_{\infty,\alpha;R_1}\big)}{\lambda},
\end{align*}
which in turn yields by choosing $\eps=\lambda/\Lambda_1\;( \leq1)$,
\begin{align}\label{Dw1}
|\nabla^2_xu(t,x)|
\lesssim \Big(\frac{\Lambda_1}{R_1\lambda}\Big)^{2}\|{u}\|_{\infty;{R_1}}
+\frac{\|f\|_{\infty,\alpha;R_1}}{\lambda}.
\end{align}
From the definitions of $\vartheta$ and $\cH(x)$, it is easy to see that
$$
R_1=\frac{R_0}{\vartheta(\sqrt{T-S}\wedge1)}\leq \frac{\Lambda_1(x)\,R_0}{\lambda(x)\,\cH(x)}=:R.
$$
Inserting this into \eqref{Dw1} yields precisely the estimate \eqref{Ds0}.

\vspace{1mm}
\noindent\textbf{Step 2. H\"older seminorm estimate \eqref{Da3} under \eqref{LL2}.}
Assume in addition that \eqref{LL2} holds with constants $c_0\in(0,1]\in[1,\infty)$.
Define
\begin{align}\label{eq:def-theta-step2}
\vartheta'
:=\Big(\frac{[a]_{L^\infty([S,T];\bC^\alpha(B_1(x)))}}
{\delta\,c_0\,\lambda(x)}\Big)^{1/\alpha}
+\frac{\|b\|_{L^\infty([S,T];\bC^\alpha(B_1(x)))}}{c_0\Lambda(x)+1}
+\frac{2}{\sqrt{T-S}\wedge 1},
\end{align}
and
\begin{align}\label{eq:def-R-step2}
R_2:=\frac{R_0}{\vartheta'(\sqrt{T-S}\wedge 1)}.
\end{align}
By construction, we have $R_2\leq \frac{R_0}2\leq \frac12$, and for fixed $y\in B_{1/2}(x)$,
$$
B_{R_2}(y)\subset B_1(x).
$$
Moreover, using \eqref{LL2}, we have
\begin{align}\label{Sq2}
\lambda(y)\geq c_0\lambda(x),\ \ \Lambda_1(y)=\Lambda(y)+1\leq (c_0^{-1}+1)\Lambda_1(x).
\end{align}
By the choice of $R_2$ in \eqref{eq:def-R-step2} and \eqref{eq:def-theta-step2},
the smallness requirements of Theorem~\ref{Schau} are satisfied for $Q_{R_2}(t,y)$.
Thus, as in Step~1, we have
$$
|\nabla_x^2 u(t,y)|
\lesssim
\Big(\frac{\Lambda_1(y)}{\lambda(y)R_2}\Big)^2\|u\|_{L^\infty(Q_{R_2}(t,y))}
+\frac{\|f\|_{L^\infty\bC^\alpha(Q_{R_2}(t,y))}}
{\lambda(y)}.
$$
Recall $R=\frac{\Lambda_1(x)\,R_0}{\lambda(x)\,\cH(x)}$. Noting that
\begin{align}\label{Sq1}
(\delta c_0) R\leq R_2=\frac{R_0}{\vartheta'(\sqrt{T-S}\wedge 1)}\leq  R,
\end{align}
we further have
$$
Q_{R_2}(t,y)=I_{R_2}(t)\times B_{R_2}(y)\subset I_{R}(t)\times B_1(x)=:\tilde Q_R(t,x)
$$
Hence,
\begin{align}\label{eq:hess-pointwise-y}
|\nabla_x^2 u(t,y)|
&\lesssim
\Big(\frac{\Lambda_1(x)}{\lambda(x)R}\Big)^2\|u\|_{L^\infty(\tilde Q_{R}(t,x))}
+\frac{\|f\|_{L^\infty\bC^\alpha(\tilde Q_{R}(t,x))}}
{\lambda(x)}.
\end{align}
This bound holds for all $y\in B_{1/2}(x)$.

\smallskip

Take $y,y'\in B_{1/2}(x)$ and consider two cases.

\smallskip
\noindent\underline{Case 1: $|y-y'|>R_2/2$.}
Using \eqref{eq:hess-pointwise-y} for $y$ and $y'$, \eqref{Sq1} and \eqref{Sq2} we obtain
\begin{align*}
&\frac{|\nabla_x^2u(t,y)-\nabla_x^2u(t,y')|}{|y-y'|^\alpha}
\leq (R_2/2)^{-\alpha}\big(|\nabla_x^2u(t,y)|+|\nabla_x^2u(t,y')|\big)\no\\
&\qquad\lesssim
\Big(\frac{\Lambda_1(x)}{\lambda(x)R}\Big)^{2+\alpha}\|u\|_{L^\infty(\tilde Q_{R}(t,x))}
+\frac{\|f\|_{L^\infty\bC^\alpha(\tilde Q_{R}(t,x))}}{\lambda(x)}.
\end{align*}

\smallskip
\noindent\underline{Case 2: $|y-y'|\leq R_2/2$.}
In this case, $y'\in B_{R_2/2}(y)$ and hence
\[
\frac{|\nabla_x^2u(t,y)-\nabla_x^2u(t,y')|}{|y-y'|^\alpha}
\leq [\nabla_x^2 u(t,\cdot)]_{\bC^\alpha(B_{R_2/2}(y))}.
\]
Applying \eqref{eq:Schauder-on-QR} to the cylinder $Q_{R_2}(t,y)$, and by \eqref{Sq1} and \eqref{Sq2},
 we get
\begin{align*}
[\nabla_x^2 u(t,\cdot)]_{\bC^\alpha(B_{R_2/2}(y))}
&\lesssim
\Big(\frac{\Lambda_1(y)}{\lambda(y)R_2}\Big)^{2+\alpha}\|u\|_{L^\infty(Q_{R_2}(t,y))}
+\frac{\|f\|_{L^\infty(Q_{R_2}(t,y))}}{\lambda(y)R_2^\alpha}
+\frac{[f]_{L^\infty\bC^\alpha(Q_{R_2}(t,y))}}{\lambda(y)}\no\\
&\lesssim
\Big(\frac{\Lambda_1(x)}{\lambda(x)R}\Big)^{2+\alpha}\|u\|_{L^\infty(\tilde Q_R(t,x))}
+\frac{\|f\|_{L^\infty(\tilde Q_{R}(t,x))}}{\lambda(x)R^\alpha}
+\frac{[f]_{L^\infty\bC^\alpha(\tilde Q_{R}(t,x))}}{\lambda(x)}.
\end{align*}
Combining the above two cases,
we obtain the desired estimate \eqref{Da3}.
\end{proof}

\section{{Pointwise gradient and Hessian estimates for SDEs}}

{In this section we combine the local PDE estimates obtained in Theorems~\ref{grad1} and~\ref{grad3}
with probabilistic representations to derive pointwise gradient and Hessian bounds for solutions of SDEs
whose coefficients may be unbounded or singular.}

\subsection{{Probabilistic solutions of Kolmogorov equations}}
{Throughout this and the next subsection, we fix $-\infty<S<T<\infty$ and a $d$-dimensional standard Brownian motion
$(W_t)_{t\in[S,T]}$.}
Consider the time-inhomogeneous SDE
\begin{align}\label{SDE}
\dif X_{t,s}(x)
=b\big(s, X_{t,s}(x)\big)\dif s+\sqrt{2}\,\sigma\big(s, X_{t,s}(x)\big)\dif W_s,
\qquad X_{t,t}(x)=x\in\mathbb{R}^{d},
\end{align}
where $S\leq t\leq s\leq T$, and
$b:[S,T]\times\mathbb{R}^d\to\mathbb{R}^d$ and
$\sigma:[S,T]\times\mathbb{R}^d\to\mathbb{R}^d\otimes\mathbb{R}^d$
are measurable.
{We always set}
\[
a(t,x):=(\sigma\sigma^*)(t,x).
\]

{We begin with a local Sobolev regularity statement for probabilistic solutions.}

\begin{theorem}\label{Th37}
Let $\alpha\in(0,1)$ and let $(q_b,p_b)\in(1,\infty]$ satisfy $\tfrac 2{q_b}+\tfrac d{p_b}<1$.
Assume that $\sigma$ is uniformly elliptic, i.e., $\lambda I_d\leq a(t,x)\leq\Lambda I_d$, and
\[
\sigma\in L^\infty\!\big([S,T];\mathbf{C}^{\alpha}(\mathbb{R}^{d})\big),
\qquad
{b\in L^{q_b}\big([S,T]; L^{p_b}(\mathbb{R}^d)\big).}
\]
Let $q\in(1,\infty)\cap(1,q_b]$ and $p\in(1,\infty)\cap(1,p_b]$ satisfy
$\tfrac d{p}+\tfrac 2{q}<2.$
Given $f\in L^q([S,T]; L^p(\mathbb{R}^d))$ and $g\in L^\infty(\mathbb{R}^d)$, define
\begin{align*}
u(t,x):= \mathbb{E}\big[ g\big(X_{t,T}(x)\big)\big]
+\mathbb{E}\!\left(\int_t^Tf\big(s,X_{t,s}(x)\big)\dif s\right).
\end{align*}
Then
{$u\in \mW^{1,2}_{q,p;\mathrm{loc}}\big((S,T)\times \mathbb{R}^{d}\big)$ solves \eqref{pde31} in the sense of Definition~\ref{Def31}.}
Moreover, for any $m\in\mathbb{N}$ and $t\in\bigl(0,\frac{T-S}2\bigr)$,
\begin{align}\label{DS4}
\|\nabla^2_xu\|_{L^q([S+t,T-t]; L^p(B_m))}
\lesssim_{C_{t,S,T,m}}
\|u\|_{L^\infty([S,T]\times B_{2m})}+\|f\|_{L^q([S,T]; L^p(B_{2m}))},
\end{align}
where $C_{t,S,T,m}>0$ depends only on $t,S,T,m,p,q,p_b,q_b,d$,
{the ellipticity constants of $\sigma$ on $[S,T]\times B_{2m}$,}
and the quantities
$\|\sigma\|_{L^\infty([S,T]; \mathbf{C}^\alpha(B_{2m}))}$ and
$\|b\|_{L^{q_b}([S,T]; L^{p_b}(B_{2m}))}$.
\end{theorem}

\begin{proof}
{\textbf{Step 1. Smooth approximation.}}
Let $\sigma_n, b_n$ and $f_n,g_n$ be smooth approximations of $\sigma,b$ and $f,g$, respectively, obtained via standard spatial mollification, so that
\[
\sigma_n, b_n,  g_n\in L^\infty\big([S,T]; C^\infty_b(\mathbb{R}^d)\big),\ \ f_n\in L^q\big([S,T]; C^\infty_b(\mathbb{R}^d)\big),
\]
and
\[
{(\sigma_n,b_n,f_n,g_n)\to(\sigma,b,f,g)\ \text{locally in the natural norms as }n\to\infty.}
\]
Denote by $X_{t,s}(x)$ and $X^n_{t,s}(x)$ the weak solutions of \eqref{SDE} corresponding to
$(\sigma,b)$ and $(\sigma_n,b_n)$, respectively.
By \cite{XZ20}, for each $S\leq t<s\leq T$ and $x\in\mathbb{R}^d$,
\begin{align}\label{DS03}
X^n_{t,s}(x)\to X_{t,s}(x)\quad\text{{in law} as }n\to\infty,
\end{align}
{and $X^n_{t,s}(x)$ admits a transition density $p_n(t,x;s,y)$ satisfying Gaussian two-sided bounds with constants depending only on the ellipticity and the local $L^{q_b}_tL^{p_b}_x$-norm of $b$; namely,}
for all $S\leq t<s\leq T$ and $x,y\in\mathbb{R}^d$,
\begin{align}\label{DS3}
\frac{C_0}{(s-t)^{d/2}}\e^{-\lambda_0|x-y|^2/(s-t)}
\leq p_n(t,x;s,y)\leq
\frac{C_1}{(s-t)^{d/2}}\e^{-\lambda_1|x-y|^2/(s-t)},
\end{align}
with constants $C_0,C_1,\lambda_0,\lambda_1$ independent of $n$.

Define $a_n:=\sigma_n\sigma_n^*$ and
\[
u_n(t,x):= \mathbb{E}\big[ g_n\big(X^n_{t,T}(x)\big)\big]
+\mathbb{E}\!\left(\int_t^Tf_n\big(s,X^n_{t,s}(x)\big)\dif s\right).
\]
{Then $u_n\in L^\infty([S,T]; C^\infty_b(\mathbb{R}^d))$ and solves \eqref{pde31} with coefficients $(a_n,b_n)$ and source term $f_n$.}

{\textbf{Step 2. Local $\mW^{2,p}$-bounds via Theorem~\ref{sobov}.}}
Fix $(t,x)\in(S,T)\times{\mathbb R}^d$. Choose
\[
{0<R<\sqrt{T-t}\wedge\sqrt{t-S}\wedge 1}
\]
small enough independent of $n$ so that $Q_R(t,x)\subset(S,T)\times{\mathbb R}^d$ and
for all $n\in\mathbb{N}$,
\[
\sup_{(s,y)\in Q_{R}(t,x)}\|a_n(s,y)-a_n(s,x)\|_{\mathrm{HS}}
\leq R^{\alpha}\sup_{s\in[S,T]}[a_n(s,\cdot)]_{\bC^\alpha(B_1(x))}
\leq \delta\,\lambda,
\]
where $\delta$ is the constant from \eqref{As1}.
Applying \eqref{essob} to $u_n$ on $Q_R(t,x)$ yields
\begin{align*}
\|\nabla^2_x u_n\|_{\mL^q_tL^p_x(Q_{R/2}(t,x))}
&\lesssim
\Big[1+\|b_n\|_{\mL^{q_b}_{p_b}(Q_{R}(t,x))}^{2/\theta_b}\Big]
\|u_n\|_{\mL^q_tL^p_x(Q_{R}(t,x))}+\|f_n\|_{\mL^q_tL^p_x(Q_{R}(t,x))},
\end{align*}
where $\theta_b:=1-\frac{2}{q_b}-\frac{d}{p_b}$ and the implicit constant depends only on
$R,p,q,p_b,q_b,d,\alpha$ and the ellipticity constants of $\sigma$ on $[t-R^2,t+R^2]\times B_2(x)$.
Since mollification does not increase local $L^{q_b}_tL^{p_b}_x$ norms,
a finite covering argument over $[S+t,T-t]\times B_m$ then yields: for any $m\in\mathbb{N}$ and $t\in(0,\frac{T-S}{2})$,
\begin{align}\label{DQ21}
\|\nabla^2_xu_n\|_{L^q([S+t,T-t];L^p(B_m))}
\lesssim \|u_n\|_{L^\infty([S,T]\times B_{2m})}+\|f\|_{L^q([S,T]; L^p(B_{2m}))},
\end{align}
with the stated dependence of constants.
Note that by \eqref{DS3} and H\"older's inequality,
\begin{align*}
\mathbb{E}\!\left(\int_t^Tf_n\big(s,X^n_{t,s}(x)\big)\dif s\right)
&\leq\int_t^T\int_{\mathbb{R}^d} f_n\big(s,y\big)\frac{C_1}{(s-t)^{d/2}}\e^{-\lambda_1|x-y|^2/(s-t)}\dif y\dif s\\
&\lesssim \int_t^T(s-t)^{-d/(2p)}\|f_n(s,\cdot)\|_{L^p}\dif s
\lesssim \|f\|_{\mL^q_tL^p_x([S,T]\times\mathbb{R}^d)},
\end{align*}
where in the last step we have used that $d/p+2/q<2$.
Thus, there is a constant $C>0$ independent of $n$ such that
$$
\|u_n\|_\infty\leq \|g\|_\infty+C\|f\|_{\mL^q_tL^p_x}.
$$
Hence, by \eqref{DQ21},
$$
\|\nabla^2_xu_n\|_{L^q([S+t,T-t];L^p(B_m))}\lesssim\|g\|_\infty+\|f\|_{\mL^q_tL^p_x([S,T]\times\mathbb{R}^d)},
$$
where the implicit constant does not depend on $n$.

{\textbf{Step 3. Passage to the limit.}}
By \eqref{DS03} and dominated convergence (using the uniform Gaussian bounds \eqref{DS3}),
{for each $(t,x)\in(S,T)\times{\mathbb R}^d$ one has $u_n(t,x)\to u(t,x)$.}
Moreover, \eqref{DQ21} implies that (up to a subsequence) for each $m\in\mathbb{N}$ and $t\in(0,\frac{T-S}{2})$,
\[
\nabla^2_xu_{n}\rightharpoonup \nabla^2_xu\quad\text{weakly in }L^q([S+t,T-t]; L^p(B_m)),
\]
and by the interpolation inequality,
\[
{\nabla_xu_n\rightarrow \nabla_xu\ \text{strongly in $L^q([S+t,T-t]; L^p(B_m))$}.}
\]
Now take $\phi\in C^\infty_c((S+t,T-t)\times B_m)$. Since $u_n$ solves the weak form of \eqref{pde31} with coefficients $(a_n,b_n)$ and source $f_n$,
\[
\langle\!\langle u_n,\partial_t\phi\rangle\!\rangle
=\langle\!\langle\sL_n u_n + f_n, \phi\rangle\!\rangle,
\]
where $\sL_n\varphi(t,x):=\tr\big(a_n(t,x)\cdot\nabla_x^2\varphi(x)\big)+b_n(t,x)\cdot\nabla_x\varphi(x)$.
{Using $a_n\to a$ and $b_n\to b$ strongly in local $L^{q_b}_tL^{p_b}_x$,
together with the weak convergence of $\nabla^2_xu_n$ and the strong convergence of $\nabla_x u_n$,
we may pass to the limit in each term and conclude that $u$ satisfies \eqref{pde31} in the sense of Definition~\ref{Def31}.}
Finally, taking weak lower semicontinuity limits in \eqref{DQ21} yields \eqref{DS4}.
\end{proof}

\begin{remark}\label{Re42}
If, in addition to the assumptions of Theorem \ref{Th37},
we also suppose that $b,f\in L^\infty([S,T];\bC^\alpha(\mathbb{R}^d))$;
then by using the Schauder estimate
in Theorem \ref{Schau},
$$
u\in L^\infty_{loc}((S,T);\bC^{2+\alpha}_{loc}(\mathbb{R}^d)),
$$
and for any $m\in\mathbb{N}$ and $t\in\bigl(0,\frac{T-S}2\bigr)$,
\begin{align*}
[\nabla^2_xu]_{L^\infty([S+t,T-t]; \bC^{2+\alpha}(B_m))}
\lesssim_{C_{t,S,T,m}}
\|u\|_{L^\infty([S,T]\times B_{2m})}+\|f\|_{L^\infty([S+t,T-t]; \bC^{\alpha}(B_{2m}))}.
\end{align*}
\end{remark}

Next we introduce a Lyapunov-type hypothesis which provides global moment control and will be used to handle unbounded coefficients on the whole space.

\vspace{1mm}
\begin{enumerate}[$({\bf \widetilde H}^\sigma_b)$]
\item  Assume $\sigma$ satisfies {\bf {(H$^\alpha_a$)}} with $a=\sigma\sigma^*$,
and $b\in L^{q_b}_{\rm loc}(\mathbb{R}; L^{p_b}_{\rm loc}(\mathbb{R}^d))$ for some $q_b,p_b\in(1,\infty]$ with
$\theta_b:=1-\frac {d}{p_b}-\frac {2}{q_b}>0$.
Moreover,
for every $(t,x)\in[S,T)\times\mathbb{R}^d$ the solution of SDE \eqref{SDE} does not explode,
and there exist two functions $\rho_0, \rho_1:\mathbb{R}^d\to(0,\infty)$ such that
\begin{align}\label{Lya1}
\sup_{S\leq t\leq s\leq T}\mathbb{E}\,\rho_0\big(X_{t,s}(x)\big)\leq \rho_1(x),\qquad x\in\mathbb{R}^d.
\end{align}
\end{enumerate}

Recall that the space $\mathcal{B}_{\rho_0}$ is defined by (\ref{WF}). We can now state the main result of this subsection.

\begin{theorem}\label{grad22}
Assume $({\bf \widetilde H}^\sigma_b)$ holds. For every $g\in\mathcal{B}_{\rho_0}$ and $f\in L^\infty([S,T]; \mathcal{B}_{\rho_0})$, define
\begin{align}\label{prob2}
u(t,x):= \mathbb{E}\big[ g\big(X_{t,T}(x)\big)\big]
+\mathbb{E}\!\left(\int_t^Tf\big(s,X_{t,s}(x)\big)\dif s\right).
\end{align}
Then for any $q\in(1,\infty)\cap(1,q_b]$ and $p\in(1,\infty)\cap(1,p_b]$
satisfying $\tfrac d{p}+\tfrac 2{q}<2$, we have
$u\in \mW^{1,2}_{q,p;\mathrm{loc}}\big((S,T)\times \mathbb{R}^{d}\big)$ and $u$ satisfies PDE \eqref{pde31}.
\end{theorem}

\begin{proof}
We argue in two steps by localizing the coefficients.

\textbf{Step 1. Bounded data.}
Assume first that $f$ and $g$ are bounded measurable.
For $n\in\mathbb{N}$, let $\chi_n\in C_c^\infty(\mathbb{R}^d)$ be nonnegative with
$\chi_n(x)=1$ for $|x|\leq n$ and $\chi_n(x)=0$ for $|x|>n+2$.
Define
\[
\sigma_n:=\sqrt{a_n}:=\sqrt{a\chi_n+(1-\chi_n)I_d},\qquad b_n:=b\chi_n.
\]
By the Lipschitz continuity of the matrix square-root on uniformly elliptic matrices, $\sigma_n$ inherits the local $C^\alpha$-regularity from $a$.
These coefficients satisfy the assumptions of Theorem~\ref{Th37}, and for $|x|\leq n$,
\begin{align}\label{NN1}
\sigma_n(t,x)=\sigma(t,x),\quad b_n(t,x)=b(t,x).
\end{align}
Let $X^n_{t,s}(x)$ be the solution to \eqref{SDE} associated with $(\sigma_n,b_n)$, and set
\[
\tau_n:=\inf\{s>t: |X^n_{t,s}(x)|\geq n\},\qquad
\hat\tau_n:=\inf\{s>t: |X_{t,s}(x)|\geq n\}.
\]
The nonexplosion condition together with \cite[p.32, Theorem~10.4]{Pin95} gives
\begin{align}\label{DS2}
\lim_{n\to\infty}\mathbb{P}(\hat\tau_n<T)=0,
\qquad
{\lim_{n\to\infty}\mathbb{P}(\tau_n<T)=0,}
\end{align}
and for each $S\leq t\leq s\leq T$, $x\in\mathbb{R}^d$ and $n\in\mathbb{N}$,
\begin{align}\label{Con1}
\mbox{ $X^n_{t,{s\wedge\tau_n}}(x)$ has the same law as $X_{t,s\wedge\hat\tau_n}(x)$.}
\end{align}

Define
\[
u_n(t,x):= \mathbb{E}\big[ g\big(X^n_{t,T}(x)\big)\big]
+\mathbb{E}\!\left(\int_t^Tf\big(s,X^n_{t,s}(x)\big)\dif s\right).
\]
Since $|u_n(t,x)|\leq \|g\|_\infty+(T-S)\|f\|_\infty$,
\eqref{NN1} and Theorem~\ref{Th37} imply that for any $m\in\mathbb{N}$ and $t\in(0,\frac{T-S}{2})$,
\begin{align}\label{DQ1}
\sup_{n\geq 2m}\|\nabla^2_xu_n\|_{L^q([S+t,T-t];L^p(B_m))}<\infty,
\end{align}
and for any $\phi\in C^\infty_c((S+t,T-t)\times B_m)$,
\begin{align}\label{DQ2}
\langle\!\langle u_n,\partial_t\phi\rangle\!\rangle=\langle\!\langle\sL_n u_n + f, \phi\rangle\!\rangle,
\end{align}
where $\sL_n\varphi(t,x):=\tr\big(a_n(t,x)\cdot\nabla_x^2\varphi(x)\big)+b_n(t,x)\cdot\nabla_x\varphi(x)$.
Moreover, for each $(t,x)\in(S,T)\times\mathbb{R}^d$, $u_n(t,x)\to u(t,x)$ pointwise.
Indeed,
\begin{align*}
&\bigl|\mathbb{E}\big[ g\big(X^n_{t,T}(x)\big)\big]-\mathbb{E}\big[ g\big(X_{t,T}(x)\big)\big]\bigr|\\
&\quad\leq \bigl|\mathbb{E}\big[ g\big(X^n_{t,T\wedge\tau_n}(x)\big)\big]
-\mathbb{E}\big[ g\big(X_{t,T\wedge\hat\tau_n}(x)\big)\big]\bigr|
+2\|g\|_\infty\bigl[\mathbb{P}(\tau_n<T)+\mathbb{P}(\hat\tau_n<T)\bigr]\\
&\quad\stackrel{\eqref{Con1}}{=}2\|g\|_\infty\bigl[\mathbb{P}(\tau_n<T)+\mathbb{P}(\hat\tau_n<T)\bigr]
\stackrel{\eqref{DS2}}{\longrightarrow}0,
\end{align*}
and similarly for the integral term involving $f$.
{Using \eqref{DQ1} and passing to the limit in \eqref{DQ2} exactly as in the proof of Theorem~\ref{Th37}, we conclude that $u\in \mW^{1,2}_{q,p;{\rm loc}}$ and solves \eqref{pde31}.}

\textbf{Step 2. General data in $\mathcal B_{\rho_0}$.}
{Let $g_n:=g\chi_n$ and $f_n:=f\chi_n$. Then $g_n$ and $f_n$ are bounded, and by Step~1 the corresponding probabilistic solution}
\[
u_n(t,x):=\mathbb{E}\big[g_n\big(X_{t,T}(x)\big)\big]
+\mathbb{E}\!\left(\int_t^Tf_n\big(s,X_{t,s}(x)\big)\dif s\right)
\]
{belongs to $\mW^{1,2}_{q,p;{\rm loc}}$ and solves \eqref{pde31} with source term $f_n$.}
Moreover, since $g\in\mathcal{B}_{\rho_0}$ and $f\in L^\infty([S,T];\mathcal{B}_{\rho_0})$,
\begin{align*}
|u_n(t,x)|
&\leq \|g\|_{\mathcal{B}_{\rho_0}}\mathbb{E}\big[ \rho_0(X_{t,T}(x))\big]
+\|f\|_{L^\infty([S,T];\mathcal{B}_{\rho_0})}
\mathbb{E}\!\left(\int_t^T\rho_0(X_{t,s}(x))\dif s\right)\\
&\stackrel{\eqref{Lya1}}{\leq} \bigl(\|g\|_{\mathcal{B}_{\rho_0}}+(T-t)\|f\|_{L^\infty([S,T];\mathcal{B}_{\rho_0})}\bigr)\,\rho_1(x).
\end{align*}
{In particular, for each fixed $m$ the family $\{u_n\}_n$ is uniformly bounded on $[S,T]\times B_{2m}$, hence the local estimate \eqref{DS4} applies uniformly (with $f_n$ in place of $f$).}
Finally, since $g_n\to g$ and $f_n\to f$ pointwise and $|g_n|\leq |g|$, $|f_n|\leq |f|$,
{dominated convergence (using \eqref{Lya1}) yields $u_n(t,x)\to u(t,x)$ pointwise.}
Passing to the limit in the weak formulation as in Step~1 gives that $u\in \mW^{1,2}_{q,p;{\rm loc}}$ solves \eqref{pde31}.
\end{proof}

\begin{remark}
{In general, for $u$ defined by \eqref{prob2} one cannot expect the identity \eqref{Id1} to extend to $s=T$ without additional terminal regularity.
Consequently, uniqueness of such probabilistic solutions for \eqref{pde31} is not addressed here.}
\end{remark}

\begin{remark}\label{Re45}
In the situation of Theorem \ref{grad22}, if we assume $b\in L^\infty([S,T];\bC^\alpha_{loc}(\mathbb{R}^d))$
and $f\equiv0$, then by using remark \ref{Re42}, one can show $u\in L^\infty_{loc}((S,T);\bC^{2+\alpha}_{loc}(\mathbb{R}^d))$.
\end{remark}

\subsection{Derivative estimates for time-inhomogeneous SDEs}

Under the hypothesis $({\bf \widetilde H}^\sigma_b)$, for any $\varphi\in\mathcal{B}_{\rho_0}$ and $t\in[S,T]$ we define
\[
\mathcal{T}_{t,T}\varphi(x):=\mathbb{E}\,\varphi\big(X_{t,T}(x)\big).
\]
A direct combination of Theorems~\ref{grad22} and \ref{grad1} yields the following short-time gradient bound.

\begin{theorem}[Short-time gradient estimate]\label{Cor43}
Assume that $({\bf \widetilde H}^\sigma_b)$ holds.
For any $q\in(1,\infty)\cap(1,q_b]$ and $p\in(1,\infty)\cap(1,p_b]$ satisfying
$
\frac{d}{p}+\frac{2}{q}<1,
$
there exists a constant $C=C(d,q,p,q_b,p_b,\alpha)>0$ such that for any $\varphi\in\mathcal{B}_{\rho_0}$ and every
$(t,x)\in\big[\tfrac{T+S}{2},T\big)\times\mathbb{R}^{d}$,
\begin{align*}
\big|\nabla_x \mathcal{T}_{t,T}\varphi(x)\big|
\lesssim_{C}\;
\frac{(\sqrt{T-S}\vee1)\,\mathcal{G}_t(x)}{\sqrt{T-t}}\,
\Lambda_1^{1/q}(x)\left(\frac{\Lambda_1(x)}{\lambda(x)}\right)^{\frac{d}{p}+\frac{2}{q}}
\sup_{s\in[S,T]}\big\|\mathcal{T}_{s,T}\varphi\big\|_{L^\infty(B_1(x))},
\end{align*}
where $\Lambda_1(x):=\Lambda(x)+1$ and $\mathcal{G}_t(x)$ is defined by \eqref{HH1}.
\end{theorem}

\begin{proof}
By Theorem~\ref{grad22}, the function $u(t,x):=\mathcal{T}_{t,T}\varphi(x)$ belongs to
$\mW^{1,2}_{q,p;\mathrm{loc}}\big((S,T)\times\mathbb{R}^d\big)$ and solves \eqref{pde31} with $f\equiv 0$.
Hence we may apply \eqref{grad0} to obtain
\begin{align}\label{Gr1}
\big|\nabla_x \mathcal{T}_{t,T}\varphi(x)\big|
\lesssim
\Lambda^{1/q}_1(x)\Big(\frac{\Lambda_1(x)}{\lambda(x)R}\Big)^{1+\frac dp+\frac2q}
\big\|\mathcal{T}_{\cdot,T}\varphi\big\|_{\mL^q_tL^p_x(Q_{R}(t,x))},
\end{align}
where $R_0:=\frac{\sqrt{T-t}}{\sqrt{T-S}\vee1}$ and
$R:=\frac{\Lambda_{1}(x)R_{0}}{\lambda(x)\mathcal{G}_t(x)}\leq R_0$.

By the definition of $\mL^q_p(Q_{R}(t,x))$ and the bound $L^\infty\hookrightarrow L^p$ on bounded sets,
\[
\big\|\mathcal{T}_{\cdot,T}\varphi\big\|_{\mL^q_p(Q_{R}(t,x))}
\lesssim R^{\frac{2}{q}+\frac{d}{p}}\,
\big\|\mathcal{T}_{\cdot,T}\varphi\big\|_{L^\infty(I_R(t)\times B_R(x))}.
\]
Since $\theta=1-\frac{d}{p}-\frac{2}{q}$, inserting this bound into \eqref{Gr1} yields
\[
\big|\nabla_x \mathcal{T}_{t,T}\varphi(x)\big|
\lesssim
\Lambda_1^{1/q}(x)\left(\frac{\Lambda_1(x)}{\lambda(x)}\right)^{\frac{d}{p}+\frac{2}{q}}
\frac{\mathcal{G}_t(x)}{R_0}\,
\sup_{s\in[S,T]}\big\|\mathcal{T}_{s,T}\varphi\big\|_{L^\infty(B_1(x))}.
\]
Finally, since $R_0^{-1}=(\sqrt{T-S}\vee1)/\sqrt{T-t}$, the desired estimate follows.
\end{proof}

Similarly, combining Remark~\ref{Re45} and \ref{grad3} yields a short-time Hessian estimate.

\begin{theorem}[Short-time Hessian estimate]\label{Cor44}
Assume that $({\bf \widetilde H}^\sigma_b)$ and \eqref{LL2} hold. Suppose further that
$b\in L^\infty\big([S,T];\mathbf{C}^\alpha_{\mathrm{loc}}(\mathbb{R}^{d})\big)$ with the same $\alpha$ as in {\bf (H$^\alpha_a$)}.
Then for $\beta\in\{0,\alpha\}$ there exists a constant $C=C(d,\alpha)>0$ such that for any $\varphi\in\cB_{\rho_0}$ and all
$(t,x)\in\big[\tfrac{T+S}{2},T\big)\times \mathbb{R}^{d}$,
\begin{align*}
\big[\nabla_{x}^{2}\mathcal{T}_{t,T}\varphi\big]_{\bC^\beta(B_{1/2}(x))}
\lesssim_{C}\left[\frac{(\sqrt{T-S}\vee1)\cH(x)}{\sqrt{T-t}}\right]^{2+\beta}
\sup_{s\in[S,T]}\big\|\mathcal{T}_{s,T}\varphi\big\|_{L^\infty(B_1(x))},
\end{align*}
where $\cH(x)$ is defined in \eqref{HH2}.
\end{theorem}

\subsection{Long time derivative estimates for time-homogeneous SDEs}

We now specialize to the time-homogeneous SDE \eqref{eq:sde}, namely,
\[
\dif X_{t}(x)=b\big(X_{t}(x)\big)\dif t+\sqrt{2}\,\sigma\big(X_{t}(x)\big)\dif W_{t},\qquad X_{0}(x)=x.
\]
Recall that the semigroup $\cT_t\varphi$ is defined by \eqref{Semi}.
Time-homogeneity allows us to combine the short-time regularization of $\cT_t\varphi$
with the ergodic behavior of $X_t(x)$ to obtain long-time derivative bounds.

We first record the short-time pointwise gradient and Hessian bounds for $\cT_t\varphi$.

\begin{theorem}\label{grad2}
Assume that {\bf (H$^\sigma_{b}$)} holds, and let $\varphi\in \mathcal{B}_{\rho_0}$.
\begin{enumerate}[(i)]
\item \textbf{(Short-time gradient estimate)}
For any $\eps\in(0,1)$ there exists a constant $C=C(\eps,d,p_b,\alpha)>0$ such that for all $(t,x)\in\mathbb{R}_+\times \mathbb{R}^{d}$,
\begin{align}\label{Hes0}
\big|\nabla_x\mathcal{T}_t\varphi(x)\big|
\lesssim_C \frac{ \Gamma_t(x) }{\sqrt{t\wedge 1}}  \Lambda_1^{\eps}(x)
\left(\frac{\Lambda_{1}(x)}{\lambda(x)}\right)^{3\eps+d/p_b}
\|\mathcal{T}_{\cdot}\varphi\|_{L^\infty([0,2t]\times B_1(x))},
\end{align}
where $\Lambda_1(x):=\Lambda(x)+1$ and $\Gamma_t(x)$ is defined by \eqref{s1}.

\item \textbf{(Short-time Hessian estimate)}
If in addition $b\in \mathbf{C}^\alpha_{\mathrm{loc}}(\mathbb{R}^d)$ and \eqref{LL2} holds, then for $\beta\in\{0,\alpha\}$ there exists a constant $C=C(d,\alpha)>0$ such that for all $(t,x)\in\mathbb{R}_+\times \mathbb{R}^{d}$,
\begin{align}\label{Hes1}
\big[\nabla^2_x\mathcal{T}_t \varphi\big]_{\bC^\beta(B_{1/2}(x))}
\lesssim_C \left(\frac{\widetilde\Gamma_t(x)}{\sqrt{t\wedge 1}}\right)^{2+\beta}
\|\mathcal{T}_{\cdot}\varphi\|_{L^\infty([0,2t]\times B_1(x))},
\end{align}
where $\widetilde\Gamma_t(x)$ is defined by \eqref{s2}.
\end{enumerate}
\end{theorem}

\begin{proof}
Suppose that $b\in L^{p_b}_{\mathrm{loc}}(\mathbb{R}^d)$ for some $p_b\in(d,\infty]$.
We apply Theorems~\ref{Cor43}--\ref{Cor44} with $S=0$ and $T=2t$.
Then
\[
R_0=\frac{\sqrt{T-t}}{\sqrt{(T-S)\vee1}}
=\frac{\sqrt{t}}{\sqrt{2t\vee1}}=\sqrt{t\wedge\tfrac12}.
\]
Since $b$ is time-independent, we have $b \in L^\infty(\mathbb{R}; L^{p_b}_{\mathrm{loc}}(\mathbb{R}^d))$, which allows us to set $q_b=\infty$.
To obtain the exponent $\frac d {p_b}+3\eps$,
we choose
$q=\tfrac1\eps$ and $\tfrac{d}{p}=\tfrac{d}{p_b}+\eps$ in Theorems~\ref{Cor43}
so that $\frac{d}{p}+\frac{2}{q}=\frac{d}{p_b}+3\eps$.
With this choice, Theorem~\ref{Cor43} yields \eqref{Hes0}
since $b\in L^{p_b}_{\mathrm{loc}}(\mathbb{R}^d)$ in {\bf (H$^\sigma_{b}$)}.
The Hessian bound \eqref{Hes1} follows analogously from Theorem~\ref{Cor44}.
\end{proof}

\begin{remark}
From \eqref{Mom1} we obtain
\begin{align*}
\|\mathcal{T}_{\cdot}\varphi\|_{L^\infty([0,2t]\times B_1(x))}
\leq\ell_0(2t)\,\|\rho_1\|_{L^\infty(B_1(x))}\,\|\varphi\|_{\cB_{\rho_0}}.
\end{align*}
Theorem~\ref{thm:main-intro} follows from Theorem~\ref{grad2}.
\end{remark}

We now prove the long-time derivative bounds in Theorem~\ref{thm:main-intro2}.

\begin{proof}[Proof of Theorem \ref{thm:main-intro2}]
Set $\widetilde \varphi:=\varphi-\mu(\varphi)$. For $t>2$, using the semigroup property and applying \eqref{Hes0} with $t=1$ we obtain
\begin{align*}
\big|\nabla_x \mathcal{T}_t \varphi(x)\big|
&=\big|\nabla_x \mathcal{T}_{t}\widetilde \varphi(x)\big|
=\big|\nabla_x \mathcal{T}_{1}\big(\mathcal{T}_{t-1} \widetilde \varphi\big)(x)\big|\\
&\lesssim \Gamma_1(x)\Lambda_1^{\eps}(x)\left(\frac{\Lambda_{1}(x)}{\lambda(x)}\right)^{3\eps+d/p_b}
\sup_{s\in[0,1]}\big\|\mathcal{T}_{s+t-1}\widetilde \varphi\big\|_{L^{\infty}(B_{1}(x))}.
\end{align*}
By the ergodic estimate \eqref{Erg1}, we have
\[
\|\mathcal{T}_{s+t-1}\widetilde \varphi\|_{L^{\infty}(B_{1}(x))}
\lesssim \ell_1(t-1)\,\|\rho_1\|_{L^\infty(B_{1}(x))}\,\|\varphi\|_{\mathcal{B}_{\rho_0}}.
\]
This yields \eqref{s3}. The Hessian estimate \eqref{s4} follows analogously from \eqref{Hes1}.
\end{proof}

In the short-time estimate \eqref{Hes1}, the singularity near $t=0$ is non-integrable.
A natural question is whether, for H\"older terminal data, one can improve the time singularity.
To this end, we introduce the following class of weights.

Let $\sW$ denote the class of functions $\rho:\mathbb{R}^d\to(0,\infty)$ such that there exist constant
$c_0\in(0,1]$ with
\begin{align}\label{Dt1}
c_0\rho(x)\leq
\inf_{y\in B_1(x)}\rho(y)\leq
\sup_{y\in B_1(x)}\rho(y)\leq c_0^{-1}\rho(x).
\end{align}
For any $\rho_1,\rho_2\in\sW$ and $\gamma\in\mathbb{R}$, it is straightforward that
\[
\rho_1+\rho_2,\quad \rho_1\rho_2,\quad \rho_1/\rho_2,\quad \rho_1^\gamma\in\sW.
\]
Given $\rho\in\sW$ and $\alpha\geq 0$, define the weighted H\"older space
\[
\bC_\rho^\alpha(\mathbb{R}^d):=\left\{f\in \bC_{loc}^\alpha(\mathbb{R}^d):\
\|f\|_{\bC^\alpha_\rho}:=
\sup_{x\in\mathbb{R}^d}\frac{\|f\|_{\bC^\alpha(B_1(x))}}{\rho(x)}<\infty\right\}.
\]
For $\alpha\in(0,1]$, we have
\[
\|f\|_{\bC^\alpha_\rho}\asymp
\sup_{x\in\mathbb{R}^d}\frac{|f(x)|}{\rho(x)}+[f]_{\bC^\alpha_\rho},
\qquad
[f]_{\bC^\alpha_\rho}:=\sup_{x\in\mathbb{R}^d}\sup_{|y|\leq1}\frac{|f(x+y)-f(x)|}{\rho(x)|y|^\alpha}.
\]

\begin{lemma}\label{Le32}
Let $\beta\in[0,1]$ and $\varphi\in\bC^\beta_\rho(\mathbb{R}^d)$.
Let $\varphi_\eps:=\varphi*\varrho_\eps$, where $\varrho_\eps(x)=\eps^{-d}\varrho(x/\eps)$ and
$\varrho\in C^\infty_c(B_1(0))$ is a smooth probability density.
Then for any $\alpha\in[0,1]$ there exists $C=C(d,\beta,\alpha,\varrho)>0$ such that for all $\eps\in(0,1)$,
\[
\|\varphi_\eps-\varphi\|_{\bC^0_\rho}\lesssim_C\eps^\beta[\varphi]_{\bC^\beta_\rho},
\qquad
\|\varphi_\eps\|_{\bC^{2+\alpha}_\rho}\lesssim_C\eps^{\beta-2-\alpha}\|\varphi\|_{\bC^\beta_\rho}.
\]
\end{lemma}

We can now prove the following improved short-time bounds.

\begin{theorem}\label{Th24}
Assume that {\bf (H$^\sigma_b$)} holds with $\rho_0,\rho_1,\lambda,\Lambda\in\sW$.
Let $1\leq\rho_2\in\sW$ and set $\rho:=\rho_0/\rho_2$.
\begin{enumerate}[(i)]
\item Suppose that $\|a\|_{\bC^\alpha(B_1(x))}+|b(x)|\leq\rho_2(x)$ for all $x$.
Then for any $\beta\in[0,1]$ there exists $\rho_3\in\sW$ such that for all $\varphi\in\bC^\beta_\rho(\mathbb{R}^d)$ and $t\in(0,1)$,
\begin{align}\label{Gf}
[\cT_t\varphi]_{\bC^\beta_{\rho_3}}\leq [\varphi]_{\bC^\beta_\rho}.
\end{align}

\item Suppose that $\|a\|_{\bC^\alpha(B_1(x))}+\|b\|_{\bC^\alpha(B_1(x))}\leq\rho_2(x)$ for all $x$.
Then for every $\beta\in[0,1]$ there exists $\rho_3\in\sW$ such that for all $\varphi\in\bC^\beta_\rho(\mathbb{R}^d)$ and $t\in(0,1)$,
\begin{align}\label{Gf1}
\|\nabla^2_x\mathcal{T}_t \varphi\|_{\bC^\alpha_{\rho_3}}
\leq t^{-1+(\beta-\alpha)/2}\|\varphi\|_{\bC^\beta_\rho}.
\end{align}
\end{enumerate}
\end{theorem}

\begin{proof}
\textbf{(i) Step 1. A basic estimate for smooth $\varphi$.}
Let $\varphi\in\bC^2_\rho(\mathbb{R}^d)$.
By the assumption $\|a\|_{\bC^\alpha(B_1(x))}+|b(x)|\leq\rho_2(x)$ and the definition $\rho=\rho_0/\rho_2$, we have
\begin{align}\label{Dq31}
\|\sL\varphi\|_{\cB_{\rho_0}}
\lesssim \|\nabla^2\varphi\|_{\cB_\rho}+\|\nabla\varphi\|_{\cB_\rho}
\leq \|\nabla\varphi\|_{\bC^1_\rho}.
\end{align}
Applying It\^o's formula to $\varphi(X_t(x))$ gives
\begin{align}\label{Dq3}
\cT_{t}\varphi(x)-\varphi(x)=\int^t_0\cT_{s}\sL \varphi(x)\dif s=:u(t,x).
\end{align}
As in the proof of Theorem~\ref{grad2}, combining Theorems~\ref{grad22} and \ref{grad1} yields that there exists $\rho_4\in\sW$ such that for all $t\in(0,1)$,
\begin{align*}
|\nabla_x u(t,x)|
&\lesssim \frac{\rho_4(x)}{\sqrt{t}}\|u\|_{L^\infty([0,2t]\times B_1(x))}
+\sqrt{t}\,\|\sL\varphi\|_{L^\infty(B_1(x))}.
\end{align*}
Moreover, by \eqref{Mom1} and \eqref{Dt1},
\begin{align}\label{Dq33}
\|u\|_{L^\infty([0,2t]\times B_1(x))}
\leq \int^{2t}_0\|\cT_{s}\sL \varphi\|_{L^\infty(B_1(x))}\dif s
\lesssim t\,\rho_1(x)\,\|\sL\varphi\|_{\cB_{\rho_0}}.
\end{align}
Combining \eqref{Dq31}--\eqref{Dq33}, we find $\rho_5\in\sW$ such that for all $t\in(0,1)$,
\[
|\nabla_x u(t,x)|\leq \sqrt{t}\,\rho_5(x)\,\|\nabla\varphi\|_{\bC^1_\rho}.
\]
Inserting this into \eqref{Dq3} yields
\begin{align}\label{Dq14}
|\nabla_x \cT_t\varphi(x)|
\leq \rho(x)\|\nabla\varphi\|_{\bC^0_\rho}
+\sqrt{t}\,\rho_5(x)\|\nabla\varphi\|_{\bC^1_\rho}.
\end{align}

\textbf{Step 2. Extend to $\varphi\in\bC^1_\rho$ by mollification.}
Fix $\varphi\in\bC^1_\rho(\mathbb{R}^d)$ and let $\varphi_\eps$ be as in Lemma~\ref{Le32}.
Using \eqref{Dq14} for $\varphi_\eps$ and the short-time gradient bound \eqref{Hes0} for $\varphi-\varphi_\eps$, there exists $\rho_6\in\sW$ such that for all $\eps\in(0,1)$ and $t\in(0,1)$,
\begin{align*}
|\nabla_x\cT_t \varphi(x)|
&\leq|\nabla_x\cT_t(\varphi-\varphi_\eps)(x)|+|\nabla_x\cT_t \varphi_\eps(x)|\\
&\leq t^{-1/2}\rho_6(x)\|\varphi-\varphi_\eps\|_{\cB_\rho}
+\rho(x)\|\nabla\varphi_\eps\|_{\bC^0_\rho}+\sqrt{t}\rho_5(x)\|\nabla\varphi_\eps\|_{\bC^1_\rho}\\
&\lesssim \Big(t^{-1/2}\rho_6(x)\eps
+\rho(x)+\sqrt{t}\rho_5(x)\eps^{-1}\Big)\|\nabla\varphi\|_{\bC^0_\rho},
\end{align*}
where we used Lemma~\ref{Le32} in the last step.
Choosing $\eps=t^{1/2}$ yields $\rho_3\in\sW$ such that
\[
\|\nabla\cT_t\varphi\|_{\cB_{\rho_3}}\leq\|\nabla\varphi\|_{\cB_\rho},\qquad t\in(0,1).
\]

\textbf{Step 3. H\"older seminorm.}
Let $|h|\leq 1$. Using the decomposition
\[
\cT_t\varphi=(\cT_t\varphi-\cT_t\varphi_\eps)+\cT_t\varphi_\eps,
\]
we obtain
\begin{align*}
|\cT_t\varphi(x+h)-\cT_t\varphi(x)|
&\leq |\cT_t\varphi_\eps(x+h)-\cT_t\varphi_\eps(x)|
+|\cT_t(\varphi-\varphi_\eps)(x+h)|+|\cT_t(\varphi-\varphi_\eps)(x)|.
\end{align*}
By the gradient bound on $\cT_t\varphi_\eps$ and Lemma~\ref{Le32},
\[
|\cT_t\varphi_\eps(x+h)-\cT_t\varphi_\eps(x)|
\lesssim |h|\,\rho_3(x)\,\|\nabla\varphi_\eps\|_{\cB_\rho}
\lesssim |h|\,\rho_3(x)\,\eps^{\beta-1}[\varphi]_{\bC^\beta_\rho}.
\]
Also, by \eqref{Mom1} and Lemma~\ref{Le32},
\[
|\cT_t(\varphi-\varphi_\eps)(x)|\lesssim \rho_1(x)\|\varphi-\varphi_\eps\|_{\cB_{\rho_0}}
\lesssim \rho_1(x)\eps^\beta[\varphi]_{\bC^\beta_\rho}.
\]
Thus,
\[
|\cT_t\varphi(x+h)-\cT_t\varphi(x)|
\lesssim\Big(|h|\,\rho_3(x)\eps^{\beta-1}+\rho_1(x)\eps^\beta\Big)[\varphi]_{\bC^\beta_\rho}.
\]
Choosing $\eps=|h|$ (valid since $|h|\leq1$) yields \eqref{Gf} for a suitable weight in $\sW$.

\vspace{1mm}
\textbf{(ii)} We proceed similarly, using the short-time Hessian bound and mollification.
By the Hessian estimate \eqref{Hes1} and \eqref{Dt1}, there exists $\rho_4\in\sW$ such that for all $t\in(0,1)$ and $\psi\in\cB_{\rho_0}$,
\begin{align}\label{Dq2}
\|\nabla^2_x\mathcal{T}_t \psi\|_{\bC^\alpha(B_{1/2}(x))}
\leq t^{-1-\alpha/2}\rho_4(x)\|\psi\|_{\cB_{\rho_0}}.
\end{align}
Let $\varphi\in\bC^2_\rho$ and define $u$ by \eqref{Dq3}.
As above, combining Theorems~\ref{grad22} and \ref{grad3} yields that for all $t\in(0,1)$,
\begin{align*}
\|\nabla^2_x u(t)\|_{\bC^\alpha(B_{1/2}(x))}
&\lesssim \left[\frac{\Lambda_1(x)\widetilde\Gamma(x)}{\lambda(x)\sqrt{t}}\right]^{2+\alpha}
\|u\|_{L^\infty([0,2t]\times B_1(x))}\\
&+\frac{\widetilde\Gamma(x)^\alpha\|\sL\varphi\|_{L^\infty(B_1(x))}}{\lambda(x) t^{\alpha/2}}
+\frac{[\sL\varphi]_{\bC^\alpha(B_1(x))}}{\lambda(x)},
\end{align*}
where
\[
\widetilde\Gamma(x):=
\frac{[a]^{1/\alpha}_{\bC^\alpha(B_1(x))}}{\lambda(x)^{1/\alpha}}
+\frac{\|b\|_{\bC^\alpha(B_1(x))}}{\Lambda_1(x)}+1.
\]
By the assumption $\|a\|_{\bC^\alpha(B_1(x))}+\|b\|_{\bC^\alpha(B_1(x))}\leq\rho_2(x)$ and $\rho=\rho_0/\rho_2$,
there exists $\rho_5\in\sW$ such that
\[
\|\sL\varphi\|_{\bC^\alpha(B_1(x))}
\leq \rho_5(x)\,\|\nabla\varphi\|_{\bC^{1+\alpha}_\rho}.
\]
Using \eqref{Dq31}--\eqref{Dq33}, we find $\rho_6,\rho_7\in\sW$ such that for all $t\in(0,1)$,
\[
\|\nabla^2_x u(t)\|_{\bC^\alpha(B_{1/2}(x))}
\leq t^{-\alpha/2}\rho_6(x)\|\nabla\varphi\|_{\bC^1_\rho}
+\rho_7(x)\|\nabla\varphi\|_{\bC^{1+\alpha}_\rho}.
\]
Inserting this into \eqref{Dq3}, we obtain $\rho_8\in\sW$ such that for all $t\in(0,1)$,
\begin{align}\label{Dq4}
\|\nabla^2_x\mathcal{T}_t \varphi\|_{\bC^\alpha(B_{1/2}(x))}
\leq t^{-\alpha/2}\rho_6(x)\|\nabla\varphi\|_{\bC^1_\rho}
+\rho_8(x)\|\nabla\varphi\|_{\bC^{1+\alpha}_\rho}.
\end{align}
Now let $\varphi\in\bC^\beta_\rho(\mathbb{R}^d)$ and let $\varphi_\eps$ be the mollification.
By \eqref{Dq2} and \eqref{Dq4},
\begin{align*}
\|\nabla^2_x\mathcal{T}_t \varphi\|_{\bC^\alpha(B_{1/2}(x))}
&\leq\|\nabla^2_x\mathcal{T}_t(\varphi-\varphi_\eps)\|_{\bC^\alpha(B_{1/2}(x))}
+\|\nabla^2_x\mathcal{T}_t \varphi_\eps\|_{\bC^\alpha(B_{1/2}(x))}\\
&\leq t^{-1-\alpha/2}\rho_4(x)\|\varphi-\varphi_\eps\|_{\cB_\rho}
+t^{-\alpha/2}\rho_6(x)\|\nabla\varphi_\eps\|_{\bC^1_\rho}
+\rho_8(x)\|\nabla\varphi_\eps\|_{\bC^{1+\alpha}_\rho}.
\end{align*}
Using Lemma~\ref{Le32} yields
\[
\|\nabla^2_x\mathcal{T}_t \varphi\|_{\bC^\alpha(B_{1/2}(x))}
\lesssim \Big(t^{-1-\alpha/2}\rho_4(x)\eps^\beta
+t^{-\alpha/2}\rho_6(x)\eps^{\beta-1}
+\rho_8(x)\eps^{\beta-1-\alpha}\Big)\|\varphi\|_{\bC^\beta_\rho}.
\]
Choosing $\eps=t^{1/2}$ gives \eqref{Gf1}.
\end{proof}

As a direct consequence of Theorems~\ref{Th24} and \ref{thm:main-intro2}, we have the following.

\begin{corollary}
Assume that {\bf (H$^\sigma_b$)$'$} holds with $\rho_0,\rho_1,\lambda,\Lambda\in\sW$ and that
for some $\rho_2\in\sW$,
\[
\|a\|_{\bC^\alpha(B_1(x))}+\|b\|_{\bC^\alpha(B_1(x))}\leq\rho_2(x),\qquad x\in\mathbb{R}^d.
\]
Let $\rho:=\rho_0/\rho_2$. For any $\beta\in[0,1]$, there exist $\rho_4,\rho_5\in\sW$ such that for all $t>0$,
\[
[\cT_t\varphi]_{\bC^{\beta}_{\rho_4}}\lesssim \|\varphi\|_{\bC^\beta_{\rho}},
\qquad
\|\nabla^2_x\cT_t\varphi\|_{\bC^\alpha_{\rho_5}}\lesssim
\begin{cases}
t^{-1+(\beta-\alpha)/2}\|\varphi\|_{\bC^\beta_{\rho}}, & t\in(0,2),\\
\ell_1(t-1)\|\varphi\|_{\bC^0_{\rho}}, & t\in[2,\infty).
\end{cases}
\]
\end{corollary}

\section{Poisson equations on the whole space}

In this section we apply the regularity estimates obtained in Sections~3 and~4
to study solvability and regularity of Poisson equations on the whole space.
More precisely, we consider
\begin{align}\label{eq:poisson}
\sL u(x)=\tr\big(a(x)\cdot\nabla^2_xu(x)\big)+b(x)\cdot\nabla_xu(x)=-f(x),
\qquad x\in\mathbb{R}^{d}.
\end{align}
We impose the following standing assumption on the coefficients.

\begin{enumerate}[{\bf (H$^\sigma_b$)$''$}]
\item Assume that {\bf (H$^\sigma_b$)$'$} holds and, in addition, the function $\ell_1$ appearing in \eqref{Erg1}
satisfies $\ell_1\in L^1([0,\infty))$.
\end{enumerate}

\vspace{1mm}
Recall that $\sL$ is the generator associated with \eqref{eq:sde}, and $\mu$ is the unique invariant probability
measure of the semigroup $(\mathcal T_t)_{t\ge0}$. In particular,
\[
\sL^*\mu=0
\quad\Longleftrightarrow\quad
\int_{\mathbb{R}^d}\sL \psi(x)\,\mu(\dif x)=0
\quad\text{for all suitable test functions }\psi.
\]
Hence a necessary solvability condition for \eqref{eq:poisson} is the centering condition
\begin{align}\label{Center}
\mu(f):=\int_{\mathbb{R}^d} f(x)\,\mu(\dif x)=0.
\end{align}
Throughout this section we assume \eqref{Center}.
We then define the candidate solution
\begin{align}\label{pou}
u(x):=\int_0^\infty \mathcal T_t  f(x)\,\dif t.
\end{align}
Under {\bf (H$^\sigma_b$)$''$} the integral in \eqref{pou} is finite for every $x$ and moreover
\[
|u(x)|
\leq \int_0^\infty |\mathcal T_t  f(x)|\,\dif t
\lesssim \rho_1(x)\|f\|_{\mathcal B_{\rho_0}}
\int_0^\infty \ell_1(t)\,\dif t
\lesssim \rho_1(x)\|f\|_{\mathcal B_{\rho_0}}.
\]
We can now state the first main result of this section.
\begin{theorem}
Assume that {\bf (H$^\sigma_b$)$''$} holds.
Let $f\in\mathcal{B}_{\rho_0}$ satisfy \eqref{Center} and define $u$ by \eqref{pou}.
Then for every $p>d$ we have $u\in \mW^{2}_{p;\mathrm{loc}}(\mathbb{R}^{d})$, and $u$ solves the Poisson equation \eqref{eq:poisson}
in the weak sense.
Moreover, for any $\eps\in(0,1)$ there exists a constant $C=C(d,p_b,\eps,\alpha,\ell_1)>0$ such that for all $x\in\mathbb{R}^d$,
\begin{align}\label{u2}
|\nabla_xu(x)|
\lesssim_{C}\;\Gamma_1(x)
\Lambda_1^{\eps}(x)
\left(\frac{\Lambda_{1}(x)}{\lambda(x)}\right)^{3\eps+d/p_b}
\|\rho_1\|_{L^\infty(B_{1}(x))}\,
\|f\|_{\mathcal{B}_{\rho_0}},
\end{align}
where $\Lambda_1(x):=\Lambda(x)+1$ and $\Gamma_1(x)$ is given by \eqref{s1} with $t=1$.
\end{theorem}

\begin{proof}
\textbf{Step 1. Gradient bound.}
Using the short-time gradient estimate for $t\in(0,1)$ and the long-time decay estimate for $t\ge1$,
we obtain
\begin{align*}
|\nabla_xu(x)|
&=\left|\nabla_x\int_0^\infty \mathcal T_t f(x)\,\dif t\right|
\leq \int_0^1 |\nabla_x\mathcal T_t f(x)|\,\dif t
+\int_1^\infty |\nabla_x\mathcal T_t f(x)|\,\dif t\\
&\lesssim\Gamma_1(x)
\Lambda_1^{\eps}(x)\left(\frac{\Lambda_1(x)}{\lambda(x)}\right)^{3\eps+d/p_b}
\|\rho_1\|_{L^\infty(B_1(x))}\,\|f\|_{\mathcal B_{\rho_0}}
\left(\int_0^1 t^{-1/2}\,\dif t+\int_1^\infty \ell_1(t)\,\dif t\right),
\end{align*}
which is finite since $\ell_1\in L^1([0,\infty))$. This yields \eqref{u2}.

\vspace{1mm}
\textbf{Step 2. Local $\mW^2_p$ regularity and identification of the Poisson equation.}
Fix an arbitrary bounded smooth domain $\cO\subset\mathbb{R}^d$.
Consider the Dirichlet problem
\begin{equation}\label{di}
\left\{
\begin{array}{ll}
\sL \hat u(x)=-f(x),& x\in \cO,\\[1mm]
\hat u=u,& x\in\partial\cO.
\end{array}\right.
\end{equation}
Since $u$ is continuous, standard elliptic theory (e.g.\ \cite[Theorem~1.2]{Kr00}) implies that for every $p>d$ there exists a unique solution
$\hat u\in \mW^{2}_{p}(\cO)\cap C(\overline\cO)$ to \eqref{di}.

Let $\tau_\cO:=\inf\{t\geq 0: X_{t}(x)\notin \cO\}$ be the first exit time from $\cO$.
Applying It\^o's formula to $\hat u(X_{t\wedge\tau_\cO}(x))$ and taking expectations, we obtain
\begin{align}\label{Ito-hat}
\hat u(x)
&= \mathbb{E}\,\hat u\big(X_{\tau_\cO}(x)\big)
-\mathbb{E}\!\left(\int_0^{\tau_\cO} \sL\hat u(X_s(x))\,\dif s\right)\no\\
&= \mathbb{E}\,u\big(X_{\tau_\cO}(x)\big)
+\mathbb{E}\!\left(\int_0^{\tau_\cO} f(X_s(x))\,\dif s\right).
\end{align}
On the other hand, by the definition of $u$ and the strong Markov property,
\begin{align}\label{Markov-u}
u(x)
&=\mathbb{E}\!\left(\int_0^{\tau_\cO}  f(X_s(x))\,\dif s\right)
+\mathbb{E}\!\left(\int_{\tau_\cO}^{\infty}  f(X_s(x))\,\dif s\right)\\
&=\mathbb{E}\!\left(\int_0^{\tau_\cO}  f(X_s(x))\,\dif s\right)
+\mathbb{E}\!\left[\int_0^\infty  \mathbb{E} f\big(X_{s}(x')\big)\,\dif s\Big|_{x'=X_{\tau_\cO}(x)}\right]\notag\\
&=\mathbb{E}\!\left(\int_0^{\tau_\cO} f(X_s(x))\,\dif s\right)
+\mathbb{E}\,u\big(X_{\tau_\cO}(x)\big).
\notag
\end{align}
Combining \eqref{Ito-hat} and \eqref{Markov-u} yields $\hat u(x)=u(x)$ for all $x\in\cO$.
As $\cO$ is arbitrary, this shows that $u\in \mW^{2}_{p;\mathrm{loc}}(\mathbb{R}^d)$ for every $p>d$ and that $u$ solves \eqref{eq:poisson}.
\end{proof}

Under strengthened regularity assumptions, we obtain a Schauder-type estimate.

\begin{theorem}
Assume that {\bf (H$^\sigma_b$)$''$} and \eqref{LL2} hold.
Suppose $b,f\in \mathbf{C}_{\mathrm{loc}}^{\alpha}(\mathbb{R}^{d})$ and $f\in\mathcal{B}_{\rho_0}$ satisfies \eqref{Center}.
Then there exists a constant $C=C(d,\alpha,\ell_0,\ell_1)>0$ such that for every $x\in\mathbb{R}^d$,
\begin{align}\label{u3}
\|u\|_{\bC^{2+\alpha}(B_{1/2}(x))}
\lesssim_C
\Bigl[\widetilde{\Gamma}_1^{\,2+\alpha}(x)\,\rho_1(x)+\rho_0(x)/\lambda(x) \Bigr]\,
\|f\|_{\bC^\alpha_{\rho_0}},
\end{align}
where $\widetilde{\Gamma}_1(x)$ is defined by \eqref{s2} with $t=1$.
\end{theorem}

\begin{proof}
For $t>0$ we split $u$ as
\[
u(x)=\int_{0}^t\mathcal{T}_s f(x)\dif s+\int^\infty_t\mathcal{T}_sf(x)\dif s=:u_0(t,x)+u_1(t,x).
\]
Take $t=1$. For the short-time part $u_0$, since by \eqref{s2} and \eqref{LL2},
\[
\|u_0\|_{L^\infty([0,2]\times B_1(x))}
\lesssim \|\rho_1\|_{L^\infty(B_1(x))}\,\|f\|_{\mathcal B_{\rho_0}}
\lesssim \rho_1(x)\|f\|_{\mathcal B_{\rho_0}},
\]
arguing as in the proof of Theorem~\ref{grad2} and employing Theorems~\ref{grad22} and \ref{grad3}, we obtain
\begin{align*}
\|\nabla^2_xu_0(1,\cdot)\|_{\bC^\alpha(B_{1/2}(x))}
&\lesssim \widetilde{ \Gamma}^{2+\alpha}_{1}(x)\|u_0\|_{L^\infty([0,2]\times B_1(x))}
+\lambda(x)^{-1}\|f\|_{\bC^\alpha(B_1(x))}\\
&\lesssim \bigl(\widetilde{ \Gamma}^{2+\alpha}_{1}(x)\,\rho_1(x)+\rho_0(x)/\lambda(x)\bigr)\|f\|_{\bC^\alpha_{\rho_0}}.
\end{align*}
For the long-time part $u_1$, estimate \eqref{s4} yields
\begin{align*}
\|\nabla^2_xu_1(1,\cdot)\|_{\bC^\alpha(B_1(x))}
&\leq \int_{1}^{\infty}\|\nabla^2_x \mathcal{T}_s f\|_{\bC^\alpha(B_1(x))}\dif s\\
&\lesssim \widetilde{ \Gamma}^{2+\alpha}_{1}(x)\|\rho_1\|_{L^\infty(B_{1}(x))}\|f\|_{\mathcal{B}_{\rho_0}}
\int_{1}^{\infty}\ell_1(s-1)\dif s.
\end{align*}
The integral is finite because $\ell_1\in L^1([0,\infty))$ under {\bf (H$^\sigma_b$)$''$}. Hence both the short-time and the long-time contributions are finite, and combining the two estimates gives \eqref{u3}.
\end{proof}

\section{Gradient estimates for SDEs with distributional drifts}

As an application of the pointwise gradient estimates derived in the previous sections,
we establish a short-time gradient bound for the Markov semigroup associated with SDEs
whose drift contains a distributional component, a class of highly singular coefficients.

For $(\alpha,p)\in(\mathbb{R}\setminus\mathbb{Z})\times[1,\infty]$, let $\mathbb{H}^\alpha_p$ denote the Bessel potential space
\[
\mathbb{H}^\alpha_p
:=\Big\{f\in \mathcal{S}'(\mathbb{R}^d): \|f\|_{\mathbb{H}^\alpha_p}:=\|(I-\Delta)^{\alpha/2}f\|_{L^p}<\infty\Big\},
\]
where $(I-\Delta)^{\alpha/2}$ is defined via Fourier transform by
\[
(I-\Delta)^{\alpha/2}f:=\Big((1+|\xi|^2)^{\alpha/2}\widehat f(\xi)\Big)^{\check{}}.
\]

Fix $d\geq2$ and consider the following SDE on $\mathbb{R}^d$:
\begin{align}\label{in:SDE}
\dif X_t = b(t,X_t)\,\dif t + \sqrt{2}\,\dif W_t,\qquad t\ge0,
\end{align}
where the drift $b=b_1+b_2$ satisfies:

\vspace{1mm}
\begin{enumerate}[{\bf (A)}]
\item
There exist $\alpha_b\in(-1,-\tfrac12]$ and $p_b,q_b\in[2,\infty]$ such that
\[
\tfrac{d}{p_b}+\tfrac{2}{q_b}<1+\alpha_b,
\qquad
b_1,\ \div b_1\in \bigcap_{T>0} L^{q_b}\big([0,T];\mathbb{H}^{\alpha_b}_{p_b}\big),
\]
and there exist constants $c_0,c_1\geq0$ such that
\[
|b_2(t,x)|\leq c_0+c_1|x|,\qquad (t,x)\in[0,\infty)\times\mathbb{R}^d.
\]
\end{enumerate}
Since $b_1$ is a distribution, the term $b_1(t,X_t)$ is not classically defined.
Following \cite{HZ25}, solutions are introduced through mollification.
Let $\phi\in C_c^\infty(\mathbb{R}^d)$ be a nonnegative mollifier with $\int\phi=1$ and set $\phi_n(x):=n^d\phi(nx)$.
Define the smooth approximation of $b_1$ by
\[
b_{1;n}(t,x):=(b_1(t,\cdot)*\phi_n)(x).
\]
Let $\mathcal P(\mathbb{R}^d)$ denote the set of Borel probability measures on $\mathbb{R}^d$.
We begin by introducing the concept of weak solutions for the SDE \eqref{in:SDE} with a distributional drift.

\begin{definition}[Weak solutions]\label{Def1}
Let $\mathfrak F=(\Omega,\mathcal F,(\mathcal F_t)_{t\ge0},\mathbb P)$ be a stochastic basis and
let $(X,W)$ be $\mathbb{R}^d$-valued, continuous, $(\mathcal F_t)$-adapted processes.
We call $(\mathfrak F,X,W)$ a weak solution of \eqref{in:SDE} with initial law $\mu\in\mathcal P(\mathbb{R}^d)$ if
$W$ is an $(\mathcal F_t)$-Brownian motion, $\mathbb P\circ X_0^{-1}=\mu$, and for all $t\in[0,T]$,
\[
X_t = X_0 + A^{b_1}_t + \int_0^t b_2(s,X_s)\,\dif s + \sqrt{2}\,W_t,\qquad\mathbb P\text{-a.s.},
\]
where
\[
A^{b_1}_t := \lim_{n\to\infty}\int_0^t b_{1;n}(s,X_s)\,\dif s \mbox{ exists in $L^2(\Omega)$.}
\]
\end{definition}

We recall the following result from \cite{HZ25}.

\begin{theorem}[Existence, uniqueness, and Krylov estimate]
Assume {\bf (A)}. Then for each $x\in\mathbb{R}^d$, there exists a weak solution $(\mathfrak F,X,W)$ of \eqref{in:SDE}
starting from $x$ in the sense of Definition~\ref{Def1}, and it is unique in the class of solutions satisfying the Krylov estimate:
for any $(\alpha,p,q)\in[\alpha_b,0]\times[2,\infty]^2$ with
\[
\tfrac{q_b}{2}\leq q\leq q_b,\qquad p\leq p_b,\qquad
\alpha-\tfrac{d}{p}-\tfrac{2}{q}\geq \alpha_b-\tfrac{d}{p_b}-\tfrac{2}{q_b},
\]
and for any $T,m>0$, there exist $\theta>0$ and $C=C(T,m,\alpha,p,q)>0$ such that for all
$f\in L^q([0,T];\mathbb H_p^\alpha)\cap L^q([0,T];C_b^\infty)$ and $0\leq t_0<t_1\leq T$,
\begin{align}\label{Kry2}
\left\|\int_{t_0}^{t_1} f(s,X_s)\,\dif s\right\|_{L^m(\Omega)}
\leq C\,(t_1-t_0)^{\frac{1+\theta}{2}}\,
\|f\|_{L^q([0,T];\mathbb H_p^\alpha)}.
\end{align}
\end{theorem}

Now we can show the following gradient estimate.
\begin{theorem}
Assume {\bf (A)}. For every $T>0$ there exists a constant $C>0$
such that
for all $\varphi\in L^\infty(\mathbb{R}^d)$ and all $(t,x)\in(0,T]\times\mathbb{R}^d$,
\[
|\nabla_x \mathbb E\,\varphi(X_t(x))|
\leq C\left(1+|x|+\frac{1}{\sqrt t}\right)\|\varphi\|_\infty.
\]
\end{theorem}

\begin{proof}
For any $T>0$,
by \cite[Section 4]{HZ25}, there is a measurable function $\Phi_t(x):[0,T]\times\mathbb{R}^d\to\mathbb{R}^d$ such that
$x\mapsto \Phi_t(x)$ forms a $C^1$-diffeomorphism and
\begin{align}\label{GG2}
|\nabla_x\Phi_t(x)|,\ |\nabla_x\Phi^{-1}_t(x)|\leq 4,\ \ t\in[0,T], x\in\mathbb{R}^d.
\end{align}
Let $X$ be the unique weak solution of \eqref{in:SDE} satisfying the Krylov estimate \eqref{Kry2},
and define the transformed process
\[
Y_t:=\Phi_t(X_t).
\]
By \cite[Lemma 4.5]{HZ25}, $Y$ solves an SDE of the form
$$
\dif Y_t=\tilde b(t,Y_t)\dif t+\sqrt2\tilde \sigma(t,Y_t)\dif W_t,
$$
where
$$
\tilde b(t,y):=\lambda (y-\Phi^{-1}_t(y))+(b_2\cdot\nabla\Phi_t)(\Phi^{-1}_t(y)),\quad \tilde\sigma(t,y):=\nabla\Phi_t(\Phi^{-1}_t(y)).
$$
Moreover, $\tilde b$ and $\tilde\sigma$ are measurable, and there exist $\gamma\in(0,1)$ and $C_T>0$ such that
for all $t\in[0,T]$ and $y,y'\in\mathbb{R}^d$,
\begin{align}\label{coef-bds}
|\tilde b(t,y)|\leq C_T(1+|y|),\qquad
|\tilde\sigma(t,y)-\tilde\sigma(t,y')|\leq C_T|y-y'|^\gamma,
\end{align}
and the diffusion matrix is uniformly elliptic:
\begin{align}\label{ell-Y}
|\xi|/4\leq |\tilde\sigma(t,y)\xi|\leq 4|\xi|,
\qquad \forall \xi\in\mathbb{R}^d.
\end{align}

Fix $t\in(0,T]$. In order to apply the backward-in-time gradient estimate from Section~4
(which is stated on a symmetric time interval around $0$), we extend the coefficients
$(\tilde b,\tilde\sigma)$ from $[0,t]$ to $[-t,t]$ by freezing them for negative times:
\[
\tilde b(s,\cdot):=\tilde b(0,\cdot),\qquad \tilde\sigma(s,\cdot):=\tilde\sigma(0,\cdot),
\qquad s\in[-t,0].
\]
This extension preserves the bounds \eqref{coef-bds}--\eqref{ell-Y} on $[-t,t]$.

\vspace{1mm}
Let $\psi_t:=\varphi\circ \Phi_t^{-1}$. Then $\|\psi_t\|_\infty=\|\varphi\|_\infty$, and
\[
\mathbb E\,\varphi(X_t(x))
=\mathbb E\,\psi_t(Y_t(\Phi_0(x))).
\]
By Theorem~\ref{Cor43} with $(S,t,T)=(-t,0,t)$, there exists $C_T>0$ such that for all $y\in\mathbb{R}^d$,
\begin{align}\label{GG1-new}
\big|\nabla_y \mathbb E\,\psi_t(Y_{0,t}(y))\big|
\leq C_T\left(1+|y|+\frac{1}{\sqrt t}\right)\|\psi_t\|_\infty
= C_T\left(1+|y|+\frac{1}{\sqrt t}\right)\|\varphi\|_\infty.
\end{align}
By the chain rule and \eqref{GG2},
\[
\nabla_x \mathbb E\,\varphi(X_t(x))
=\nabla_y \Big(\mathbb E\,\psi_t(Y_{0,t}(y))\Big)\Big|_{y=\Phi_0(x)}\cdot \nabla_x\Phi_0(x),
\]
hence
\[
|\nabla_x \mathbb E\,\varphi(X_t(x))|
\leq \|\nabla_x\Phi_0\|_\infty\,
\big|\nabla_y \mathbb E\,\psi_t(Y_{0,t}(y))\big|_{y=\Phi_0(x)}.
\]
Combining \eqref{GG1-new} with \eqref{GG2} and using $|\Phi_0(x)|\leq 4|x|+C_T$, we obtain
the desired estimate.
\end{proof}


\begin{thebibliography}{99}

\bibitem{BGL14}
D.~Bakry, I.~Gentil, and M.~Ledoux,
\emph{Analysis and Geometry of Markov Diffusion Operators},
Springer, Cham, 2014.

\bibitem{Be-Lo}
M.~Bertoldi and L.~Lorenzi,
Estimates of the derivatives for parabolic operators with unbounded coefficients,
\emph{Trans. Amer. Math. Soc.} \textbf{357} (2005), 2627--2664.

\bibitem{Bi84}
J.-M.~Bismut,
\emph{Large Deviations and the Malliavin Calculus},
Birkh\"auser, Boston, MA, 1984.

\bibitem{CC18}
G. Cannizzaro and K. Chouk,
Multidimensional SDEs with singular drift and universal construction of the polymer measure with white noise potential,
\emph{Ann. Probab.} \textbf{46} (2018), 1710--1763.

\bibitem{Ce}
S.~Cerrai,
\emph{Second Order PDEs in Finite and Infinite Dimensions. A Probabilistic Approach},
Lecture Notes in Math., vol.~1762, Springer, Berlin, 2001.

\bibitem{Ce2}
S.~Cerrai,
Elliptic and parabolic equations in $\mathbb{R}^n$ with coefficients having polynomial growth,
\emph{Comm. Partial Diff. Equ.} \textbf{21} (1996), 281--317.

\bibitem{Cr-Do-Go-Ot-So}
D.~Crisan, P.~Dobson, B.~Goddard, M.~Ottobre, and I.~Souttar,
Poisson equations with locally Lipschitz coefficients and uniform in time averaging for stochastic differential equations via strong exponential stability,
\emph{Ann. Inst. Henri Poincar\'e Probab. Stat.} (to appear),
preprint (2022), arXiv:2204.02679.

\bibitem{Cr-Do-Ot}
D.~Crisan, P.~Dobson, and M.~Ottobre,
Uniform in time estimates for the weak error of the Euler method for SDEs and a pathwise approach to derivative estimates for diffusion semigroups,
\emph{Trans. Amer. Math. Soc.} \textbf{374} (2021), 3289--3330.

\bibitem{Cr-Ot}
D.~Crisan and M.~Ottobre,
Pointwise gradient bounds for degenerate semigroups (of UFG type),
\emph{Proc. Roy. Soc. A} \textbf{472} (2016), 20160442, 25 pp.

\bibitem{Eb-Gu-Zi}
A.~Eberle, A.~Guillin, and R.~Zimmer,
Quantitative Harris-type theorems for diffusions and McKean--Vlasov processes,
\emph{Trans. Amer. Math. Soc.} \textbf{371} (2019), 7135--7173.

\bibitem{EL94}
K.~D.~Elworthy and X.-M.~Li,
Formulae for the derivatives of heat semigroups,
\emph{J. Funct. Anal.} \textbf{125} (1994), 252--286.

\bibitem{FFRS}
A.~Fiorenza, M.~R.~Formica, T.~Roskovec, and F.~Soudsk\'y,
Detailed proof of classical Gagliardo--Nirenberg interpolation inequality with historical remarks,
\emph{Z. Anal. Anwend.} \textbf{40} (2021), 217--236.

\bibitem{FlemingSoner06}
W.~H.~Fleming and H.~M.~Soner,
\emph{Controlled Markov Processes and Viscosity Solutions},
2nd ed., Springer, New York, 2006.

\bibitem{GP24}
L. Gr\"afner and N. Perkowski,
Weak well-posedness of energy solutions to singular SDEs with supercritical distributional drift,
arXiv:2407.09046 (2024).

\bibitem{Ha-Li}
Q.~Han and F.~Lin,
\emph{Elliptic Partial Differential Equations},
2nd ed., Amer. Math. Soc., Providence, RI, 2011.

\bibitem{HWZ20}
Z.~Hao, M.~Wu, and X.~Zhang,
Schauder estimates for nonlocal kinetic equations and applications,
\emph{J. Math. Pures Appl. (9)} \textbf{140} (2020), 139--184.

\bibitem{HZ25}
Z.~Hao and X.~Zhang,
SDEs with supercritical distributional drifts,
\emph{Commun. Math. Phys.} \textbf{406} (2025), 250.

\bibitem{Khas68}
R.~Z.~Khasminskii,
On the principle of averaging for parabolic and elliptic differential equations,
(1968).

\bibitem{Kr96}
N.~V.~Krylov,
\emph{Lectures on Elliptic and Parabolic Equations in H\"older Spaces},
Amer. Math. Soc., Providence, RI, 1996.

\bibitem{Kr08}
N.~V.~Krylov,
\emph{Lectures on Elliptic and Parabolic Equations in Sobolev Spaces},
Amer. Math. Soc., Providence, RI, 2008.


\bibitem{KrylovVMO06}
N.~V.~Krylov,
Parabolic and elliptic equations with VMO coefficients,
\emph{Comm. Partial Diff. Equ.} \textbf{32} (2007), 453--475.

\bibitem{Kr00}
N.~V.~Krylov,
Linear and fully nonlinear elliptic equation with $L_d$ drift,
\emph{Comm. Partial Diff. Equ.} \textbf{45} (2020), 1778--1798.


\bibitem{Kr-Ro}
N.~V.~Krylov and M.~R\"ockner,
Strong solutions of stochastic equations with singular time dependent drift,
\emph{Probab. Theory Related Fields} \textbf{131} (2005), 154--196.

\bibitem{Ku-Lo-Lu}
M.~Kunze, L.~Lorenzi, and A.~Lunardi,
Nonautonomous Kolmogorov parabolic equations with unbounded coefficients,
\emph{Trans. Amer. Math. Soc.} \textbf{362} (2010), 169--198.

\bibitem{LSU68}
O.~A.~Ladyzhenskaya, V.~A.~Solonnikov, and N.~N.~Ural'tseva,
\emph{Linear and Quasilinear Equations of Parabolic Type},
Amer. Math. Soc., Providence, RI, 1968.

\bibitem{Lieberman96}
G.~M.~Lieberman,
\emph{Second Order Parabolic Differential Equations},
World Scientific, River Edge, NJ, 1996.

\bibitem{LWX}
Y.~Li, F.~Wu, and L.~Xie,
Poisson equation on Wasserstein space and diffusion approximations for McKean--Vlasov equation,
\emph{SIAM J. Math. Anal.} \textbf{56} (2024), 1495--1542.

\bibitem{Lu}
A.~Lunardi,
Schauder theorems for linear elliptic and parabolic problems with unbounded coefficients in $\mathbb{R}^{N}$,
\emph{Studia Math.} \textbf{128} (1998), 171--198.

\bibitem{MST}
J.~C.~Mattingly, A.~M.~Stuart, and M.~V.~Tretyakov,
Convergence of numerical time-averaging and stationary measures via Poisson equations,
\emph{SIAM J. Numer. Anal.} \textbf{48} (2010), 552--577.

\bibitem{Me-Tw}
S.~P.~Meyn and R.~L.~Tweedie,
Stability of Markovian processes III: Foster--Lyapunov criteria for continuous-time processes,
\emph{Adv. in Appl. Probab.} \textbf{25} (1993), 518--548.

\bibitem{PP}
G.~Pag\`es and F.~Panloup,
Ergodic approximation of the distribution of a stationary diffusion: rate of convergence,
\emph{Ann. Appl. Probab.} \textbf{22} (2012), 1059--1100.

\bibitem{PV01}
\'E.~Pardoux and A.~Yu.~Veretennikov,
On the Poisson equation and diffusion approximation. I,
\emph{Ann. Probab.} \textbf{29} (2001), 1061--1085.

\bibitem{PV03}
\'E.~Pardoux and A.~Yu.~Veretennikov,
On the Poisson equation and diffusion approximation. II,
\emph{Ann. Probab.} \textbf{31} (2003), 1166--1192.

\bibitem{PavStu08}
G.~A.~Pavliotis and A.~M.~Stuart,
\emph{Multiscale Methods: Averaging and Homogenization},
Springer, New York, 2008.

\bibitem{Pin95}
R.~G.~Pinsky,
\emph{Positive Harmonic Functions and Diffusion},
Cambridge Univ. Press, Cambridge, 1995.

\bibitem{Platen99}
E.~Platen,
An introduction to numerical methods for stochastic differential equations,
\emph{Acta Numer.} \textbf{8} (1999), 197--246.

\bibitem{PW06}
E.~Priola and F.-Y.~Wang,
Gradient estimates for diffusion semigroups with singular coefficients,
\emph{J. Funct. Anal.} \textbf{236} (2006), 244--264.

\bibitem{Ro-Xi1}
M.~R\"ockner and L.~Xie,
Diffusion approximation for fully coupled stochastic differential equations,
\emph{Ann. Probab.} \textbf{49} (2021), 1205--1236.

\bibitem{TalayTubaro90}
D.~Talay and L.~Tubaro,
Expansion of the global error for numerical schemes solving stochastic differential equations,
\emph{Stoch. Anal. Appl.} \textbf{8} (1990), 483--509.

\bibitem{Th97}
A.~Thalmaier,
On the differentiation of heat semigroups and Poisson integrals,
\emph{Stochastics Stochastics Rep.} \textbf{61} (1997), 297--321.

\bibitem{TW98}
A.~Thalmaier and F.-Y.~Wang,
Gradient estimates for harmonic functions on regular domains in Riemannian manifolds,
\emph{J. Funct. Anal.} \textbf{155} (1998), 109--124.

\bibitem{WZ25}
F.-Y.~Wang and X.~Zhao,
Bismut formula and gradient estimates for Dirichlet semigroups with application to singular killed DDSDEs,
preprint (2025), arXiv:2506.19429.

\bibitem{XZ20}
L.~Xie and X.~Zhang,
Ergodicity of stochastic differential equations with jumps and singular coefficients,
\emph{Ann. Inst. Henri Poincar\'e Probab. Stat.} \textbf{56} (2020), 175--229.

\bibitem{XZ26}
L.~Xie and X.~Zhang,
Uniform-in-time diffusion approximation for multi-scale stochastic systems,
preprint.

\bibitem{ZZZ22}
X. Zhang, R. Zhu and X. Zhu,
Singular HJB equations with applications to KPZ on the real line,
\emph{Probab. Theory Related Fields} \textbf{183} (2022), 789--869.

\end{thebibliography}
\end{document}